\documentclass[a4paper]{article}
\usepackage{mathtools}
\usepackage{nicefrac}
\usepackage{algpseudocode}
\usepackage{subfig}
\usepackage{amsthm}
\usepackage{amssymb}
\usepackage{amsmath}
\usepackage{hyperref}
\usepackage{a4wide}

\hypersetup{colorlinks}

\newcommand\real{\mathbb{R}}
\newcommand\nat{\mathbb{N}}
\newcommand{\vect}[1]{\boldsymbol{#1}}
\newcommand\mesh{\mathcal{T}_{h}}
\newcommand\face[1][]{\mathcal{F}_{h}^{\mathcal{#1}}}
\newcommand\MESH{\mathcal{T}_{H}}
\newcommand\FACE[1][]{\mathcal{F}_{H}^{\mathcal{#1}}}
\newcommand\ele{\kappa}
\newcommand\ELE{K}
\newcommand{\COVERELE}{\mathcal{K}}
\DeclarePairedDelimiter{\abs}{\lvert}{\rvert}
\DeclarePairedDelimiter{\norm}{\lVert}{\rVert}
\newcommand\dgspace{V_{hp}(\mesh,\vect{p})}
\newcommand\DGSPACE{V_{HP}(\MESH,\vect{P})}
\newcommand{\Ah}{A_{hp}}
\newcommand{\Fh}{F_{hp}}
\newcommand{\AH}{A_{HP}}
\newcommand{\AV}{\widetilde{A}_{HP}}
\newcommand{\FH}{F_{HP}}
\newcommand{\uh}{u_{hp}}
\newcommand{\uH}{u_{HP}}
\newcommand{\ut}{u_{2G}}
\newcommand{\vh}{v_{hp}}
\newcommand{\vH}{v_{HP}}
\newcommand{\wH}{w_{HP}}
\newcommand{\ip}{\sigma_{hp}}
\newcommand{\IP}{\sigma_{HP}}
\newcommand{\dx}{\, \mathrm{d}\vect{x}}
\newcommand{\ds}{\, \mathrm{d}s}
\newcommand{\avg}[1]{\{\!\!\{ #1 \}\!\!\}}
\newcommand{\jump}[1]{[\![ #1 ]\!]}
\newcommand\LProj{\vect{\Pi}_{L^2}}
\DeclareMathOperator{\diam}{diam}

\newcommand{\configfigure}{\captionsetup[subfloat]{farskip=0pt,captionskip=5pt}\centering}

\theoremstyle{plain}
\newtheorem{theorem}{Theorem}[section]
\newtheorem{lemma}[theorem]{Lemma}
\newtheorem{corollary}[theorem]{Corollary}
\theoremstyle{remark}
\newtheorem{remark}[theorem]{Remark}
\theoremstyle{definition}
\newtheorem{assumption}[theorem]{Assumption}
\newtheorem{definition}[theorem]{Definition}
\newtheorem{algo}[theorem]{Algorithm}

\newcommand{\labelref}[2]{#1~\ref{#2}}
\newcommand{\secref}[1]{\labelref{Sect.}{#1}}
\newcommand{\lemmaref}[1]{\labelref{Lemma}{#1}}
\newcommand{\theoremref}[1]{\labelref{Theorem}{#1}}
\newcommand{\algoref}[1]{\labelref{Algorithm}{#1}}
\newcommand{\assumpref}[1]{\labelref{Assumption}{#1}}
\newcommand{\defref}[1]{\labelref{Definition}{#1}}
\newcommand{\figref}[1]{\labelref{Fig.}{#1}}
\newcommand{\subfigref}[2]{\figref{#1}\subref{#1:#2}}
	
\title{Two-grid $hp$-version discontinuous Galerkin finite element methods for quasilinear elliptic PDEs on agglomerated coarse meshes}
\author{Scott Congreve\thanks{Charles University, Faculty of Mathematics and Physics, Department of Numerical Mathematics, Sokolovsk\'a 83,  18675 Prague, Czech Republic. \texttt{congreve@karlin.mff.cuni.cz}} \and %
Paul Houston\thanks{School of Mathematical Sciences, University of Nottingham, University Park, Nottingham, NG7 2RD, UK. \texttt{Paul.Houston@nottingham.ac.uk}}}
\date{}
\begin{document}
\maketitle
    
\begin{abstract}
This article considers the extension of two-grid $hp$-version discontinuous Galerkin finite element methods for the numerical approximation of second-order quasilinear elliptic boundary value problems of monotone type to the case when agglomerated polygonal/polyhedral meshes are employed for the coarse mesh approximation. We recall that within the two-grid setting, while it is necessary to solve a nonlinear problem on the coarse approximation space, only a linear problem must be computed on the original fine finite element space. In this article, the coarse space will be constructed by agglomerating elements from the original fine mesh. Here, we extend the existing \emph{a priori} and \emph{a posteriori} error analysis for the two-grid $hp$-version discontinuous Galerkin finite element method from \cite{CongreveTG} for coarse meshes consisting of standard element shapes to include arbitrarily agglomerated coarse grids. Moreover, we develop an $hp$-adaptive two-grid algorithm to adaptively design the fine and coarse finite element spaces; we stress that this is undertaken in a fully automatic manner, and hence can be viewed as blackbox solver. Numerical experiments are presented for two- and three-dimensional problems to demonstrate the computational performance  of the proposed $hp$-adaptive two-grid method.
\\

\\

\end{abstract}

\section{Introduction}

In recent years there has been a tremendous amount of interest in the mathematical development and application of discretisation methods for the numerical approximation of partial differential equations (PDEs) which employ general polygonal/polyhedral (collectively referred to as polytopic) meshes; see, for example, \cite{dg_cfes_2012,Bassi2014,MimeticBook2014,Cangiani,cangiani2013hp,UnifiedVEM,di2015hybrid,Fries:Belytschko:2009,hackbusch_sauter_cfe_nm}, and the references cited therein. The exploitation of such general  meshes is highly advantageous for the efficient approximation of localized geometrical features present in the underlying geometry. Indeed, in many geophysical and engineering applications, the numerical study of fluid-structure interaction, crack and wave propagation phenomena, and flow in fractured porous media, for example, are typically characterized by a strong complexity of the physical domain, cf. \cite{Antonietti2021}. Furthermore, the ability to utilise polytopic meshes offers a number of advantages also in the context of multilevel linear solvers, such as Schwarz-based domain decomposition preconditioners and multigrid solvers, see, for example, \cite{AnHoHuSaVe2017,AnPe2018,AHPS_2020}, and the references cited therein. In this context, embedded coarse meshes can very easily be constructed using mesh agglomeration techniques. Here, collections of elements present in the original fine mesh are `glued' together to form coarse polytopic elements; a very simple approach to define these coarse elements is to employ graph partitioning software, such as METIS \cite{Karypis1999}, for example. 

In the present paper we consider the application of polytopic meshes to design black box two-grid methods for the numerical approximation of nonlinear PDE problems. Two-grid methods were originally introduced by \cite{Xu_92,Xu.J1994,Xu.J1996b} in the context of standard Galerkin finite element methods for both non-symmetric linear and nonlinear problems, cf., for example, \cite{Axelsson_Layton_1996,Bi2007,BiGinting,Wheeler98atwo-grid, Marion_Xu_1995,Utnes_1997,Wu_Allen_1999}, and the references cited therein. Extensions to discontinuous Galerkin methods (DGFEMs) have been undertaken in \cite{BiGinting,CongreveNonNewtonTG,CongreveTG}, for example. Indeed, in our own previous work \cite{CongreveNonNewtonTG,CongreveTG} we have studied both the \emph{a priori} and \emph{a posteriori} error analysis for the two-grid variant of the $hp$-version interior penalty DGFEM for the numerical solution of strongly monotone second order quasilinear elliptic partial differential equations on so called standard meshes; by this we mean meshes comprising of simplicial, quadrilateral, and hexahedral elements. 

We recall that the construction of two-grid methods for nonlinear PDE problems relies on the definition of a coarse finite element space $X$ and fine space $Y$, where the coarse space is, hopefully, considerably coarser compared to the fine space. The method first solves the (expensive) full nonlinear problem on the coarse space $X$ before utilizing this solution to linearize the underlying PDE problem on the fine space $Y$. Obtaining a solution on the fine space then only requires solving a \emph{linear} problem, which is computationally cheaper than solving the full nonlinear problem. Previous work on two-grid methods generally assume that coarse and fine meshes can be easily constructed based on employing standard shaped elements in such a manner that $X\subset Y$; i.e., that the coarse mesh is embedded within the fine one. In practice, it may be necessary to construct a coarse mesh from an unstructured fine mesh, in which case, this condition may be difficult to satisfy when standard elements are employed. To that end, we shall consider development of two-grid methods whereby the coarse mesh is constructed by agglomerating fine mesh elements. In this way, the agglomerated coarse elements will consist of general polytopic elements. In particular, in this article, we study the $hp$-version of the two-grid incomplete interior penalty DGFEM. Here we generalise the analysis presented in \cite{CongreveTG} to the current setting, whereby the coarse mesh may be constructed via agglomeration. Moreover, we develop a general purpose black box two-grid adaptive algorithm which automatically refines both the fine and coarse spaces to ensure that the discretisation error is controlled in a computationally efficient manner.

The outline of this article is as follows. In \secref{section:tg} we introduce the underlying model problem, together with its two-grid $hp$-version DGFEM approximation. Next, in \secref{section:error_analysis} we derive an \emph{a priori} error bound for the proposed numerical scheme. \secref{sec:apost_and_adaptivity} is devoted to the development of $hp$-version adaptive algorithms for automatically refining both the coarse and fine meshes. In \secref{section:numerics}, we perform numerical experiments to demonstrate the performance of the proposed adaptive strategy. Finally, in \secref{sec:concluding_remarks} we summarise the work presented in this article and discuss potential extensions.

\section{Model problem and two-grid \texorpdfstring{$hp$}{hp}-version DGFEM} 
\label{section:tg}

In this section we introduce a model second-order quasilinear elliptic boundary value problem and discuss its numerical approximation based on employing the two-grid DGFEM on a fine mesh comprising of standard  elements, but with a coarse mesh consisting of elements constructed by agglomerating fine elements.

\subsection{Model problem}
In this article we consider the numerical approximation of the following quasilinear elliptic boundary value problem: find $u\in H^1(\Omega)$ such that
\begin{equation}
\label{eqn:quasilinear_eqn}
\begin{aligned}
-\nabla\cdot(\mu(\vect{x},\abs{\nabla u})\nabla u) &= f(\vect{x}) && \text{in }\Omega, \\
u &= 0 && \text{on }\Gamma.
\end{aligned}
\end{equation}
where $\Omega$ is a bounded polygonal/polyhedral Lipschitz domain in $\Omega\subset\real^d$, $d=2,3$ with boundary $\Gamma\coloneqq\partial\Omega$ and $f\in L^2(\Omega)$. 
\begin{assumption}
We assume that the nonlinearity $\mu$ satisfies the following conditions:
\begin{description}
\item[(A1)] $\mu\in C^0(\overline\Omega \times [0,\infty))$ and \label{eqn:mu_continuous}
\item[(A2)] there exists positive constants $m_{\mu}$ and $M_{\mu}$ such that the following monotonicity property is satisfied:
\begin{equation}
\label{eqn:monotonicity}
m_{\mu}(t-s) \leq \mu(\vect{x},t)t - \mu(\vect{x},s)s \leq M_{\mu}(t-s), \qquad t\geq s \geq 0, \vect{x}\in\overline\Omega.
\end{equation}
\end{description}
\end{assumption}
From \cite[Lemma 2.1]{Liu1994} we note that, as $\mu$ satisfies \eqref{eqn:monotonicity}, there exists constants $C_1$ and $C_2$, $C_1\geq C_2 > 0$, such that for all vectors $\vect{v}, \vect{w}\in\real^d$ and all $\vect{x}\in\overline\Omega$,
\begin{align}
\abs{\mu(\vect{x},\abs{\vect{v}})\vect{v}-\mu(\vect{x},\abs{\vect{w}})\vect{w}} &\leq C_1 \abs{\vect{v}-\vect{w}}, \label{eqn:nl_bound:upper}\\
C_2 \abs{\vect{v}-\vect{w}}^2 & \leq (\mu(\vect{x},\abs{\vect{v}})\vect{v}-\mu(\vect{x},\abs{\vect{w}})\vect{w}) \cdot (\vect{v}-\vect{w}). \label{eqn:nl_bound:lower}
\end{align}
For ease of notation we shall suppress the dependence of $\mu$ on $\vect{x}$ and write $\mu(t)$ instead of $\mu(\vect{x},t)$.

\subsection{Meshes, spaces, and trace operators}
Following \cite{Congreve_PhD} we consider a \emph{fine} mesh $\mesh$ which partitions $\Omega\subset\real^d$, $d=2,3$, into disjoint open-element domains $\ele$ such that $\overline\Omega = \bigcup_{\ele\in\mesh} \overline{\ele}$. We assume shape-regularity of the mesh and that each element $\ele\in\mesh$ is an affine image of a reference element $\hat \ele$; i.e, for each $\ele\in\mesh$ there exists an affine mapping $T_{\ele} : \hat \ele \to \ele$ such that $\ele=T_{\ele}(\hat \ele)$, where $\hat \ele$ is the open cube $(-1,1)^3$ in $\real^3$ and either the open triangle $\{(x,y) : -1 < x < 1, -1 < y < -x \}$ or the open square $(-1,1)^2$ in $\real^2$. We denote by $h_{\ele}$ the element diameter of $\ele\in\mesh$ and $\vect{n}_{\ele}$ signifies the unit outward normal vector to $\ele$. We allow the meshes to be 1-irregular, i.e., each face of any one element $\ele\in\mesh$ contains at most one hanging node and each edge of each face contains at most one hanging node. Under this assumption we can construct an auxiliary 1-irregular mesh by subdividing all quadrilateral and hexahedral elements $\ele\in\mesh$ whose edges contain at least one hanging node into $2^d$ sub-elements. We assume that triangular elements are \emph{regularly reducible}, cf. \cite{Ortner2007}, to eliminate hanging nodes in triangular elements on the auxiliary mesh.  
We point out that these conditions are necessary to ensure that Theorem~\ref{theorem:posteriori} holds, cf. \cite{CongreveTG}. Additionally, we note that these assumptions imply that the family $\{\mesh\}_{h>0}$ is of \emph{bounded local variation}, i.e., there exists a constant $\rho_1\geq 1$, independent of element sizes, such that $\rho_1^{-1} \leq \nicefrac{h_{\ele}}{h_{\ele'}} \leq \rho_1$, for any pair of elements $\ele,\ele'\in\mesh$ which share a common face $F=\partial \ele \cap \partial \ele'$.

For a non-negative integer $p$, we denote by $\mathcal{P}_p(\hat \ele)$ the space of polynomials of total degree $p$ on $\hat \ele$. When $\hat \ele$ is a hypercube, we also consider $\mathcal{Q}_p(\hat \ele)$, the set of all tensor-product polynomials on $\hat \ele$ of degree $p$ in each coordinate direction. To each $\ele\in\mesh$ we assign a polynomial degree $p_{\ele}$ and construct the vector $\vect{p} = \{ p_{\ele} : \ele\in\mesh \}$. We suppose that $\vect{p}$ is also of bounded local variation, i.e., there exists a constant $\rho_2\geq 1$, independent of the element sizes and $\vect{p}$, such that, for any pair of neighboring elements $\ele,\ele'\in\mesh$, $\rho_2^{-1} \leq \nicefrac{p_{\ele}}{p_{\ele'}} \leq \rho_2$.
With this notation, we introduce the \emph{fine} finite element space
\[
\dgspace = \{ v\in L^2(\Omega) : v\vert_{\ele} \circ T_{\ele} \in \mathcal{R}_{p_{\ele}}(\hat \ele), \ele\in\mesh \},
\]
where $\mathcal{R}$ is either $\mathcal{P}$ or $\mathcal{Q}$.

Next, we introduce a \emph{coarse} mesh partition $\MESH$, consisting of general polytopes $\ELE$ constructed by agglomerating elements $\ele\in\mesh$ from the fine mesh, such that $\overline{\Omega}=\bigcup_{\ELE\in\MESH}\overline{\ELE}$; i.e., for all $\ELE\in\MESH$, there exists a set $\mesh(\ELE) \subseteq\mesh$ of fine mesh elements such that $\overline{\ELE} = \bigcup_{\ele\in\mesh(\ELE)} \overline{\ele}$, $\mesh = \bigcup_{\ELE\in\MESH} \mesh(\ELE)$, and, for all $\ele\in\mesh$,
\[
\ele\in\mesh(\ELE) \quad \implies \quad \ele\not\in\mesh(\ELE')\quad\forall \ELE'\in\MESH\setminus\{\ELE\}.
\]
We denote by $H_{\ele}$ the diameter of the coarse element $K\in\MESH$; i.e., $H_K=\diam(K)$. To each $\ELE\in\MESH$ we assign a polynomial degree $P_{\ELE}$ and construct the vector $\vect{P} = \{ P_{\ELE} : \ELE\in\MESH \}$. We also assume that the polynomial degree of the coarse mesh element is less than or equal to the polynomial degree of its constituent fine mesh elements; i.e., $P_{\ELE} \leq p_{\ele}$ for all $\ELE\in\MESH$ and $\ele\in\mesh(\ELE)$. We can then introduce the finite element space on the coarse mesh by
\[
\DGSPACE = \{ v\in L^2(\Omega) : v\vert_{\ELE} \in \mathcal{P}_{P_{\ELE}}(\ELE), \ELE\in\MESH \}.
\]
We note that due to the assumptions on the polynomial degree that $\DGSPACE\subseteq\dgspace$.

We shall now define some suitable face operators required for the definition of the proposed DGFEM. To this end we denote by $\face[I]$ the set of all interior faces of the fine mesh partition $\mesh$ of $\Omega$, and by $\face[B]$ the set of all boundary faces of the fine mesh $\mesh$. Additionally, we let $\face=\face[I]\cup\face[B]$ denote the set of all faces in the mesh $\mesh$. We similarly denote by $\FACE[I]$, $\FACE[B]$, and $\FACE=\FACE[I]\cup\FACE[B]$ the faces on the coarse mesh $\MESH$ following \cite{Cangiani}. Due to the construction of the coarse mesh via agglomeration of fine mesh elements we note that $\FACE\subset\face$, $\FACE[I]\subset\face[I]$, and $\FACE[B]\subset\face[B]$.

Let $v$ and $\vect{q}$ be scalar- and vector-valued functions, respectively, which are smooth inside each element $\ele\in\mesh$. Given two adjacent elements, $\ele^+,\ele^-\in\mesh$ which share a common face $F\in\face[I]$, i.e., $F=\partial\ele^+\cap\partial\ele^-$, we write $v^\pm$ and $\vect{q}^\pm$ to denote the traces of the functions $v$ and $\vect{q}$, respectively, on the faces $F$, taken from the interior of $\ele^\pm$, respectively. Using this notation, we define the averages of $v$ and $\vect{q}$ at $\vect{x}\in F$ by
\[
\avg{v} = \frac12(v^++v^-), \qquad
\avg{\vect{q}}=\frac12(\vect{q}^++\vect{q}^-),
\]
respectively. Similarly, we define the jumps of $v$ and $\vect{q}$ at $\vect x\in F$ as
\[
\jump{v} = v^+\vect{n}_{\ele^+}+v^-\vect{n}_{\ele^-}, \qquad 
\jump{\vect q} = \vect{q}^+\cdot\vect{n}_{\ele^+}+\vect{q}^-\cdot\vect{n}_{\ele^-},
\]
respectively. On a boundary face $F\in\face[B]$, we set $\avg{v}=v$, $\avg{\vect q}=\vect q$, $\jump{v}=v\vect{n}$, and $\jump{\vect q}=\vect{q}\cdot\vect{n}$, where $\vect n$ denotes the unit outward normal vector on $\Gamma$. We define $\avg{\,\cdot\,}$ and $\jump{\,\cdot\,}$ analogously on $\FACE$.

For a face $F\in\face$ of the fine mesh, we define $h_F$ to be the diameter of the face and the face polynomial degree $p_F$ to be defined by
\[
p_F =
\begin{cases}
    \max(p_{\ele},p_{\ele'}), & \text{if } F=\partial \ele \cap \partial \ele' \in\face[I], \\
    p_{\ele},& \text{if } F=\partial \ele \cap \Gamma \in\face[B].
\end{cases}
\]

\subsection{DGFEM discretization}

With the notation introduced in the previous section we first define, for comparison with the proposed two-grid method (see below), 
the following \emph{standard} DGFEM, based on employing an incomplete interior penalty formulation, on the fine space $\dgspace$, for the numerical approximation of the problem 
\eqref{eqn:quasilinear_eqn}: find $\uh\in\dgspace$ such that
\begin{equation}
\label{eqn:standard_fem}
\Ah(\uh;\uh,\vh) = \Fh(\vh)
\end{equation}
for all $\vh\in\dgspace$, where
\begin{align*}
\Ah(\phi;u,v) &= \sum_{\ele\in\mesh} \int_{\ele} \mu(\abs{\nabla\phi})\nabla u\cdot\nabla v\dx
    -\sum_{F\in\face} \int_F \avg{ \mu(\abs{\nabla_h\phi}) \nabla_h u}\cdot\jump{v} \ds \\
    &\qquad+ \sum_{F\in\face} \int_F \ip \jump{u}\cdot\jump{v}\ds, \\
\Fh(v) &= \sum_{\ele\in\mesh} \int_{\ele} fv\dx,
\end{align*}
and $\nabla_h$ is used to denote the broken gradient operator, defined element-wise. Here, the \emph{fine grid interior penalty parameter} $\ip$ is defined as 
\begin{equation}
\label{eqn:ip:fine}
\ip=\gamma_{hp} \frac{p_F^2}{h_F},
\end{equation}
where $\gamma_{hp}>0$ is a sufficiently large constant; cf. \lemmaref{lemma:coercivity} below.

On the class of spaces $H^1(\Omega)+\dgspace$, we introduce the following DGFEM \emph{energy norm}
\begin{equation}\label{eqn:norm}
\norm{v}_{hp}^2 = \norm{\nabla_h v}_{L^2(\Omega)}^2 + \sum_{F\in\face} \int_F \ip \abs{\jump{v}}^2 \ds.
\end{equation}
Following \cite{CongreveTG}, we recall that the form $\Ah(\phi;\cdot,\cdot)$, $\phi\in\dgspace$, is coercive, in the sense that the following lemma holds.
\begin{lemma}
\label{lemma:coercivity}
There exists a positive constant $\gamma_{\min}$, such that for any $\gamma_{hp}>\gamma_{\min}$, there exists a coercivity constant $C_c > 0$, independent of $h$ and $\vect{p}$, such that
\[
\Ah(\phi;v,v) \geq C_c\norm{v}_{hp}^2
\]
for all $\phi,v\in\dgspace$.
\end{lemma}

Finally, we introduce the following $hp$-version of the two-grid DGFEM approximation to \eqref{eqn:quasilinear_eqn} based on employing the fine and coarse finite element spaces $\dgspace$ and $\DGSPACE$, respectively, cf. \cite[Algorithm 1]{BiGinting} and \cite[Sect. 2.3]{CongreveTG}:
\begin{enumerate}
\item (Nonlinear solve) Compute the coarse grid approximation $\uH\in\DGSPACE$ such that
    \begin{equation}
    \label{eqn:tg_fem:coarse}
    \AH(\uH; \uH,\vH) = \FH(\vH)
    \end{equation}
    for all $\vH\in\DGSPACE$.
\item (Linear solve) Determine the fine grid solution $\ut\in\dgspace$ such that
    \begin{equation}
    \label{eqn:tg_fem:fine}
    \Ah(\uH;\ut,\vh) = \Fh(\vh)
    \end{equation}
    for all $\vh\in\dgspace$.
\end{enumerate}
Here, we note that $\FH(\cdot)$ is defined analogously to $\Fh(\cdot)$ and $\AH(\cdot; \cdot,\cdot)$ is defined on the coarse mesh partition $\MESH$ analogously to $\Ah(\cdot; \cdot,\cdot)$, but with a different \emph{coarse mesh interior penalty parameter} $\IP$; the definition of $\IP$ is given below in \eqref{eqn:ip:coarse}.

\begin{remark}
	While the above DGFEM formulation is based on employing the incomplete interior penalty method, we stress that the proceeding error analysis naturally generalises to other DGFEMs commonly employed within the literature.
\end{remark}

\subsection{Inverse estimates and approximation results for the coarse space}

Before embarking on the error analysis of the two-grid DGFEM defined in \eqref{eqn:tg_fem:coarse}--\eqref{eqn:tg_fem:fine}, we first derive some preliminary results. In particular, we revisit some inverse estimates and polynomial approximation results which are valid on general polytopic elements and are hence required to analyze the coarse grid approximation \eqref{eqn:tg_fem:coarse}. To this end, we first introduce the necessary definitions and assumptions from \cite[Section 3.2 \& 4.3]{Cangiani} required for the inverse inequality \lemmaref{lemma:inverse_ineq}.

\begin{definition}\label{def:element_simplices}
For each element $\ELE\in\MESH$ we define the family $\mathcal{F}_{\flat}^{\ELE}$ of all possible $d$-dimensional simplices contained in $\ELE$ and having at least one face in common with $\ELE$. Moreover, we write $\ELE_{\flat}^{F}$ to denote a simplex belonging to $\mathcal{F}_{\flat}^{\ELE}$ which shares with $\ELE\in\MESH$ the specific face $F\subset\partial\ELE$.
\end{definition}
\begin{assumption}\label{assumption:coarse:shape_regular}
For any $\ELE\in\MESH$, there exists a set of non-overlapping $d$-dimensional simplices $\{\ELE_{\flat}^F\}\subset\mathcal{F}_\flat^{\ELE}$ contained within $\ELE$, such that for all $F\subset\partial\ELE$, the following condition holds
\[
h_{\ELE} \leq C_s \frac{d\abs{\ELE_{\flat}^F}}{\abs{F}},
\]
where $C_s$ is a positive constant, which is independent of the discretization parameters, the number of faces that the element possesses, and the measure of $F$.
\end{assumption}

Equipped with these definitions we state the following inverse inequality from \cite[Lemma 32]{Cangiani}.

\begin{lemma}\label{lemma:inverse_ineq}
Let $\ELE\in\MESH$; then, for all $F\subset\partial\ELE$, assuming \assumpref{assumption:coarse:shape_regular} is satisfied, the inverse inequality
\[
\norm{v}_{L^2(\partial K)}^2 \leq C_s C_{\textrm{INV}}d\frac{P^2}{H_{\ELE}} \norm{v}_{L^2(\ELE)}^2
\]
holds, for each $v\in\mathcal{P}_P(\ELE)$, where $C_s$ is the constant from \assumpref{assumption:coarse:shape_regular} which is independent of $\nicefrac{\abs{\ELE}}{\sup_{\ELE_{\flat}^F\subset\ELE}\abs{\ELE_{\flat}^F}}$, $\abs{F}$, $P$, and $v$; moreover, $C_{\textrm{INV}}$ is a positive constant arising from a standard inverse inequality on simplices, and is independent of $v$, $P$, and $H_{\ELE}$.
\end{lemma}

We now turn our attention to the derivation of suitable $hp$--version approximation results for the coarse finite element space. To this end, we define a covering for the coarse mesh as follows; cf. \cite[Definition 17]{Cangiani}.
\begin{definition}\label{def:covering}
We define the \emph{covering} $\MESH^\sharp=\{\COVERELE\}$ related to the coarse mesh $\MESH$ as a set of open shape-regular $d$-simplices $\COVERELE$, such that, for each $\ELE\in\MESH$, there exists a $\COVERELE\in\MESH^\sharp$, such that $\ELE\subset\COVERELE$. 
\end{definition}

With this notation, we make the following assumption.
\begin{assumption}\label{assumption:coarse:covering}
We assume there exists a covering such that $h_{\COVERELE} \coloneqq \diam(\COVERELE) \leq C_D H_{\ELE}$, for each pair $\ELE\in\MESH$, $\COVERELE\in\MESH^\sharp$, with $\ELE\subset\COVERELE$, for a constant $C_D>0$, uniformly with respect to the mesh size.
\end{assumption}

Furthermore, we require the definition of the following classical extension operator, cf. \cite{Stein}.
\begin{theorem} \label{thm-extension}
Let $D$ be a domain with minimally smooth boundary $\partial D$. Then, there exists a linear operator $\mathfrak{E} : H^s(D)\to H^s(\real^d), s \in\nat_0$, such that $\mathfrak{E}v\vert_D =v$ and
\[
    \norm{\mathfrak{E}v}_{H^s(\real^d)} \leq C_{\mathfrak{E}} \norm{v}_{H^s(D)},
\]
where $C_\mathfrak{E}$ is a positive constant depending only on $s$ and parameters which characterize the
boundary $\partial D$.
\end{theorem}

With this notation we state the following approximation result from \cite[Lemmas 23 \& 33]{Cangiani}.
\begin{lemma}\label{lemma:coarse:approximation}
Let $\ELE\in\MESH$ and $\COVERELE\in\MESH^\sharp$ be the corresponding simplex, such that $\ELE\subset\COVERELE$, cf. \defref{def:covering}. Suppose that $v\in L^2(\Omega)$ is such that $\mathfrak{E}v\vert_{\COVERELE}\in H^{L_{\ELE}}(\mathcal{K)}$, for some $L_{\ELE}\geq 0$. Then, given \assumpref{assumption:coarse:covering} is satisfied, there exists $\Pi v \in \DGSPACE$, such that $\Pi v\vert_K \in \mathcal{P}_{P_K}(K)$ and the following bound holds
\[
\norm{v-\Pi v}_{H^q(\ELE)} \leq C_{I,1} \frac{H_{\ELE}^{S_{\ELE}-q}}{P_{\ELE}^{L_{\ELE}-q}}\norm{\mathfrak{E}v}_{H^{L_{\ELE}}(\COVERELE)}, \qquad L_{\ELE} \geq 0,
\]
for $0 \leq q \leq L_{\ELE}$. Furthermore, if $v\in H^1(\Omega)$ and \assumpref{assumption:coarse:shape_regular} is satisfied then the following bound also holds
\[
\norm{v-\Pi v}_{L^2(\partial\ELE)} \leq C_{I,2} \frac{H_{\ELE}^{S_{\ELE}-\nicefrac12}}{P_{\ELE}^{L_{\ELE}-\nicefrac12}}\norm{\mathfrak{E}v}_{H^{L_{\ELE}}(\COVERELE)}, \qquad L_{\ELE} > \nicefrac12.
\]
Here, $S_{\ELE} = \min(P_{\ELE}+1, L_{\ELE})$ and $C_{I,1}$ and $C_{I,2}$ are positive constants which depend on the shape-regularity of $\COVERELE$ and the constant $C_s$ from \assumpref{assumption:coarse:shape_regular}, but are independent of $v$, $H_{\ELE}$, and $P_{\ELE}$.
\end{lemma}

From the proof of Theorem~\ref{thm-extension} one can establish that the constant 
$C_{\mathfrak{E}}$ is independent of the measure of the underlying domain $D$, cf. \cite{AHPS_2020}.
Hence, employing Theorem~\ref{thm-extension}, Lemma~\ref{lemma:coarse:approximation} may be stated in the following simplified form.
\begin{corollary}\label{corollary:coarse:approximation}
	Under the assumptions of Lemma~\ref{lemma:coarse:approximation}, the following bounds hold
	\begin{align*}
\norm{v-\Pi v}_{H^q(\ELE)} 
&\leq C_{I,1}^\prime \frac{H_{\ELE}^{S_{\ELE}-q}}{P_{\ELE}^{L_{\ELE}-q}}\norm{v}_{H^{L_{\ELE}}(\ELE)}, && L_{\ELE} \geq 0, \qquad 0 \leq q \leq L_{\ELE}, \\
\norm{v-\Pi v}_{L^2(\partial\ELE)} 
&\leq C_{I,2}^\prime \frac{H_{\ELE}^{S_{\ELE}-\nicefrac12}}{P_{\ELE}^{L_{\ELE}-\nicefrac12}}\norm{v}_{H^{L_{\ELE}}(\ELE)}, && L_{\ELE} > \nicefrac12.
\end{align*}
\end{corollary}

Hence, the condition employed in \cite{Cangiani,cangiani2013hp} regarding the amount of overlap of the simplices $\mathcal{K}$ is not required.

\subsection{Stability analysis of the coarse grid approximation}

Based on the inverse inequality stated in \lemmaref{lemma:inverse_ineq}, following \cite[Lemma 35]{Cangiani}, we define the coarse mesh interior penalty parameter $\IP$ as follows:
\begin{equation}
\label{eqn:ip:coarse}
\IP = \begin{cases}
    \gamma_{HP} \max_{\ELE\in\{\ELE^+,\ELE^-\}} \left( C_{\textrm{INV}} \frac{P_{\ELE}^2}{H_{\ELE}}\right), & \vect{x}\in F\in\FACE[I], F=\partial\ELE^+\cap\partial\ELE^-, \\
    \gamma_{HP} C_{\textrm{INV}} \frac{P_{\ELE}^2}{H_{\ELE}}, & \vect{x}\in F\in\FACE[B], F=\partial\ELE^+\cap\Gamma,    
\end{cases}
\end{equation}
where $C_{\textrm{INV}}(P_K,K)$ is the constant from the inverse inequality \lemmaref{lemma:inverse_ineq} and $\gamma_{HP}>0$ is a sufficiently large constant; cf. \lemmaref{lemma:coarse:monotonicity_continuity}.
We also define the DGFEM norm $\norm{\cdot}_{HP}$ on the coarse mesh analogously to the fine mesh norm $\norm{\cdot}_{hp}$ from \eqref{eqn:norm} using the coarse mesh interior penalty parameter $\IP$ defined above.

To analyze the $hp$--version DGFEM defined on the coarse mesh, cf. \eqref{eqn:tg_fem:coarse}, in the case when general polytopic elements are employed, without introducing unnecessary regularity assumptions on the analytical solution $u$ and only utilizing the $hp$-approximation results available on polytopic elements, the analysis presented in \cite{CongreveTG} must be generalized in a suitable manner. To this end, we introduce the following extension of the form $\AH(\cdot; \cdot,\cdot)$, cf. \cite{Cangiani, Perugia2002}, to $\mathcal{V}\times\mathcal{V}$, where $\mathcal{V}=H^1(\Omega)+\DGSPACE$:
\begin{align*}
    \AV(u,v) = &\sum_{\ELE\in\MESH} \int_{\ELE} \mu(\abs{\nabla u})\nabla u\cdot\nabla v\dx + \sum_{F\in\FACE} \int_F \IP \jump{u}\cdot\jump{v}\ds \\*
    &-\sum_{F\in\FACE} \int_F \avg{ \mu(\abs{\LProj(\nabla_h u)}) \LProj(\nabla_h u)}\cdot\jump{v} \ds.
\end{align*}
Here, $\LProj : [L^2(\Omega)]^d \to [\DGSPACE]^d$ denotes the orthogonal $L^2$-projection onto the finite element space $[\DGSPACE]^d$. We note, that
\[
    \AV(u,v) = \AH(u; u,v) \qquad \text{for all } u,v \in \DGSPACE.
\]
The Lipschitz continuity and strong monotonicity for the extended form $\AV(\cdot,\cdot)$ are determined in the following lemma.
\begin{lemma}\label{lemma:coarse:monotonicity_continuity}
Let $\gamma_{HP}>C_1^2C_2^{-1}C_sd\epsilon$, where $\epsilon>\nicefrac14$; then, given \assumpref{assumption:coarse:shape_regular} holds, we have that the nonlinear form
$\AV(\cdot,\cdot)$ is strongly monotone in the sense that 
\begin{equation}\label{lemma:coarse:monotonicity_continuity:monotone}
\AV(v_1,v_1-v_2)-\AV(v_2,v_1-v_2) \geq C_{\mathrm{mono}}\norm{v_1-v_2}_{HP}^2 \quad \text{for all } v_1,v_2\in\mathcal{V},
\end{equation}
and Lipschitz continuous in the sense that
\begin{equation}\label{lemma:coarse:monotonicity_continuity:lipschitz}
\abs{\AV(v_1,w)-\AV(v_2,w)} \leq C_{\mathrm{cont}}\norm{v_1-v_2}_{HP}\norm{w}_{HP} \quad \text{for all } v_1,v_2,w\in\mathcal{V},
\end{equation}
where $C_{\mathrm{mono}}$ and $C_{\mathrm{cont}}$ are positive constants independent of the discretization parameters.
\end{lemma}
\begin{proof}
The following proof proceeds in a similar manner to \cite[Lemma 27]{Cangiani} with modifications to account for the nonlinearity. Starting with the bound \eqref{lemma:coarse:monotonicity_continuity:monotone}, we note by applying
\eqref{eqn:nl_bound:lower} that
\begin{align}
&\AV(v_1,v_1-v_2)-\AV(v_2,v_1-v_2) \nonumber\\*
&\quad\geq C_2 \sum_{\ELE\in\MESH}\norm{\nabla(v_1-v_2)}_{L^2(\ELE)}^2 +\sum_{F\in\FACE}\norm{\IP^{\nicefrac12}\jump{v_1-v_2}}_{L^2(F)}^2 \nonumber\\*
&\qquad\begin{multlined}[t][0.8\textwidth]-\sum_{F\in\FACE} \int_F \abs{\avg{ \mu(\abs{\LProj(\nabla_h v_1)}) \LProj(\nabla_h v_1) \nonumber\\*- \mu(\abs{\LProj(\nabla_h v_2)}) \LProj(\nabla_h v_2) }\cdot\jump{v_1-v_2}} \ds \end{multlined}\nonumber\\*
&\quad\equiv \mbox{I}+\mbox{II}+\mbox{III}.\label{lemma:coarse:monotonicity_continuity:monotone:step1}
\end{align}
We now proceed by considering Term III. To this end, for $F\in\FACE[I]$, such that $F\subset \partial K^+ \cap \partial K^-$, upon application of \eqref{eqn:nl_bound:upper}, the Cauchy-Schwarz inequality, and the arithmetic-geometric mean inequality $ab\leq  a^2\epsilon + \nicefrac{b^2}{(4\epsilon)}$, $\epsilon>0$, we deduce that
\begin{align*}
& \int_F \abs{\avg{ \mu(\abs{\LProj(\nabla_h v_1)}) \LProj(\nabla_h v_1) - \mu(\abs{\LProj(\nabla_h v_2)}) \LProj(\nabla_h v_2) }\cdot\jump{v_1-v_2}} \ds \\*
&\quad\leq C_1 \int_F \avg{\abs{\LProj(\nabla_h (v_1-v_2))}} \abs{\jump{v_1-v_2}} \ds \\*
&\quad\leq C_1^2\epsilon \left( \sum_{\ELE\in\{\ELE^+,\ELE^-\}} \norm{\IP^{-\nicefrac12} \LProj(\nabla_h (v_1-v_2)\vert_{\ELE})}_{L^2(F)}^2 \right) \\*
&\qquad+ \frac{1}{8\epsilon}\norm{\IP^{\nicefrac12}\jump{v_1-v_2}}_{L^2(F)}^2.
\end{align*}
Analogously, for $F\in\FACE[B]$, we have that
\begin{multline*}
\int_F \abs{\avg{ \mu(\abs{\LProj(\nabla_h v_1)}) \LProj(\nabla_h v_1) - \mu(\abs{\LProj(\nabla_h v_2)}) \LProj(\nabla_h v_2) }\cdot\jump{v_1-v_2}} \ds \\ 
\leq C_1^2 \epsilon \norm{\IP^{-\nicefrac12} \LProj(\nabla_h (v_1-v_2))}_{L^2(F)}^2 + \frac{1}{4\epsilon}\norm{\IP^{\nicefrac12}\jump{v_1-v_2}}_{L^2(F)}^2.
\end{multline*}
Combining these results, employing the inverse inequality \lemmaref{lemma:inverse_ineq}, the definition of $\IP$ from \eqref{eqn:ip:coarse}, and the $L^2$-stability of $\LProj$, we have that
\begin{align*}
\mbox{III}
&\leq C_1^2 \epsilon \sum_{\ELE\in\MESH} \IP^{-1} \norm{\LProj(\nabla (v_1-v_2))}_{L^2(\partial\ELE)}^2 + \frac{1}{4\epsilon}\sum_{F\in\FACE}\norm{\IP^{\nicefrac12}\jump{v_1-v_2}}_{L^2(F)}^2 \\*
&\leq \frac{C_1^2 \epsilon C_s d}{\gamma_{HP}} \sum_{\ELE\in\MESH} \norm{\nabla (v_1-v_2)}_{L^2(\ELE)}^2 + \frac{1}{4\epsilon}\sum_{F\in\FACE}\norm{\IP^{\nicefrac12}\jump{v_1-v_2}}_{L^2(F)}^2.
\end{align*}
Inserting this result into \eqref{lemma:coarse:monotonicity_continuity:monotone:step1} gives
\begin{align*}
\AV(v_1,v_1-v_2)-\AV(v_2,v_1-v_2) \geq &\left(C_2-\frac{C_1^2 \epsilon C_s d}{\gamma_{HP}}\right) \sum_{\ELE\in\MESH}\norm{\nabla(v_1-v_2)}_{L^2(\ELE)}^2
\\&+\left(1-\frac{1}{4\epsilon}\right) \sum_{F\in\FACE}\norm{\IP^{\nicefrac12}\jump{v_1-v_2}}_{L^2(F)}^2.
\end{align*}
Therefore, the nonlinear form $\AV(\cdot,\cdot)$ is strongly monotone over $\mathcal{V}\times\mathcal{V}$, assuming that $\epsilon>\nicefrac14$ and $\gamma_{HP}>C_1^2C_2^{-1}C_sd\epsilon$.

We now consider the second bound \eqref{lemma:coarse:monotonicity_continuity:lipschitz}. By applying \eqref{eqn:nl_bound:upper} and Cauchy-Schwarz, we get that
\begin{equation*}\begin{aligned}
&\abs{\AV(v_1,w)-\AV(v_2,w)} \\*
&\leq C_1 \sum_{\ELE\in\MESH}\norm{\nabla(v_1-v_2)}_{L^2(\ELE)}\norm{\nabla w}_{L^2(\ELE)}\\*
&\quad+\sum_{F\in\FACE}\norm{\IP^{\nicefrac12}\jump{v_1-v_2}}_{L^2(F)}\norm{\IP^{\nicefrac12}\jump{w}}_{L^2(F)} \\*
&\quad+ \sum_{F\in\FACE} \int_F \abs{\avg{ \mu(\abs{\LProj(\nabla_h v_1)}) \LProj(\nabla_h v_1) - \mu(\abs{\LProj(\nabla_h v_2)}) \LProj(\nabla_h v_2) }\cdot\jump{w}} \ds.
\end{aligned}\end{equation*}
Following a similar proof to the bound for \eqref{lemma:coarse:monotonicity_continuity:monotone}, without the need to employ the arithmetic-geometric mean inequality, we deduce that
\begin{multline*}
\sum_{F\in\FACE}\int_F \abs{\avg{ \mu(\abs{\LProj(\nabla_h v_1)}) \LProj(\nabla_h v_1) - \mu(\abs{\LProj(\nabla_h v_2)}) \LProj(\nabla_h v_2) }\cdot\jump{w}} \ds \\ 
\leq C_1\left( \frac{C_sd}{\gamma_{HP}} \sum_{\ELE\in\MESH}\norm{\nabla(v_1-v_2)}_{L^2(\ELE)}^2\right)^{\nicefrac12}
\left(\sum_{F\in\FACE} \norm{\IP^{\nicefrac12}\jump{w}}_{L^2(F)}^2\right)^{\nicefrac12}.
\end{multline*}
A simple application of Cauchy-Schwarz completes the proof of Lipschitz continuity.
\end{proof}

\begin{theorem}
Suppose that $\gamma_{hp}$ and $\gamma_{HP}$ are sufficiently large; cf. \lemmaref{lemma:coercivity} and \lemmaref{lemma:coarse:monotonicity_continuity}, respectively. Then, there exists a unique solution $\ut\in\dgspace$ to the two-grid DGFEM \eqref{eqn:tg_fem:coarse}--\eqref{eqn:tg_fem:fine}.
\end{theorem}
\begin{proof}
Given that \lemmaref{lemma:coarse:monotonicity_continuity} demonstrates Lipschitz continuity and strong monotonicity of the semi-linear form $\AV(\cdot,\cdot)$, and 
\[
\AV(\uH,\vH) = \AH(\uH; \uH,\vH) = \FH(\vH)
\]
for all $\vH \in \DGSPACE$, we can follow the proof of \cite[Theorem 2.5]{Houston_QL_priori} to show that $\uH$ is a unique solution of \eqref{eqn:tg_fem:coarse}. Furthermore, as the fine grid formulation \eqref{eqn:tg_fem:fine} is an interior penalty discretization of a linear elliptic PDE, where the coefficient $\mu(\abs{\nabla_h \uH})$ is a known function, the existence and uniqueness of the solution $\ut$ to this problem follows immediately; cf., for example, \cite{StammWihler,Wihler2003}.
\end{proof}

\section{Error analysis}\label{section:error_analysis}
In this section, we derive an \emph{a priori} error bound for the two-grid DGFEM \eqref{eqn:tg_fem:coarse}--\eqref{eqn:tg_fem:fine}. To this end, we first establish an \emph{a priori} error bound for the coarse mesh approximation defined by \eqref{eqn:tg_fem:coarse}.

\subsection{Coarse mesh \emph{a priori} error bound}
We first state and prove the following abstract error bound, derived in a similar manner to \cite[Lemma 4.8]{Cangiani}.
\begin{lemma}\label{lemma:abstract:priori}
Let $u\in H^1(\Omega)$ be the weak solution to \eqref{eqn:quasilinear_eqn} and $\uH\in\DGSPACE$ be the coarse mesh approximation defined by \eqref{eqn:tg_fem:coarse}. Assuming that $\gamma_{HP}$ is sufficiently large, cf.  \lemmaref{lemma:coarse:monotonicity_continuity}, the following abstract error bound holds.
\begin{multline*}
\norm{u-\uH}_{HP} \leq \left(1+\frac{C_{\mathrm{cont}}}{C_{\mathrm{mono}}}\right) \inf_{\vH\in\DGSPACE}\norm{u-\vH}_{HP} \\
    + \frac1{C_{\mathrm{mono}}} \sup_{\wH\in\DGSPACE\setminus\{0\}} \frac{\abs{\AV(u,\wH)-\FH(\wH)}}{\norm{\wH}_{HP}}.
\end{multline*}
\begin{proof}
The proof follows almost identically to \cite[Lemma 28]{Cangiani}, using the strong monotonicity \eqref{lemma:coarse:monotonicity_continuity:monotone} and Lipschitz continuity \eqref{lemma:coarse:monotonicity_continuity:lipschitz} from \lemmaref{lemma:coarse:monotonicity_continuity}.
\end{proof}
\end{lemma}
We now state the following $hp$--version \emph{a priori} error bound for the coarse mesh approximation defined by \eqref{eqn:tg_fem:coarse}.
\begin{theorem}\label{theorem:coarse_priori}
Let $\MESH$ be a coarse mesh, constructed by agglomerating elements from a shape-regular fine mesh $\mesh$, satisfying Assumptions \ref{assumption:coarse:shape_regular} and \ref{assumption:coarse:covering}, with $\MESH^{\sharp}=\{\COVERELE\}$ an associated covering of $\MESH$ consisting of $d$-simplices; cf. \defref{def:covering}. Let $\uH\in\DGSPACE$ be the coarse mesh approximation defined by \eqref{eqn:tg_fem:coarse}. If the analytical solution $u\in H^1(\Omega)$ to \eqref{eqn:quasilinear_eqn} satisfies $u\vert_{\ELE}\in H^{L_{\ELE}}(\ELE)$, $L_{\ELE} \geq \nicefrac32$, for $\ELE\in\MESH$, such that $\mathfrak{E}u\vert_{\COVERELE}\in H^{L_{\ELE}}(\COVERELE)$, where $\COVERELE\in\MESH^{\sharp}$ with $\ELE\subset\COVERELE$; then,
\[
\norm{u-\uH}_{HP}^2 \leq C_3 \sum_{\ELE\in\MESH} \frac{H_{\ELE}^{2S_{\ELE}-2}}{P_{\ELE}^{2L_{\ELE}-2}} (1+\mathcal{G}_{\ELE}(H_{\ELE},P_{\ELE})) \norm{u}_{H^{L_{\ELE}}(\ELE)}^2,
\]
where $S_{\ELE} = \min(P_{\ELE}+1,L_{\ELE})$,
\[
    \mathcal{G}_{\ELE}(H_{\ele},P_{\ele})\coloneqq \frac{P_{\ELE}+P_{\ELE}^2}{H_{\ELE}}\max_{F\subset\partial\ELE} \left.\IP^{-1}\right\vert_F + \frac{H_{\ELE}}{P_{\ELE}}\max_{F\subset\partial\ELE} \left.\IP\right\vert_F
\]
and $C_3$ is a positive constant, independent on the discretization parameters, but dependent on the constants $m_{\mu}$, $M_{\mu}$, $C_1$, and $C_2$ from the monotonicity properties of $\mu(\cdot)$.
\end{theorem}
\begin{proof}
Following the proof in \cite{Cangiani}, upon application of Corollary~\ref{corollary:coarse:approximation} we get
\begin{align} \label{theorem:coarse_priori:proof:step1}
\inf_{\vH\in\DGSPACE}&\norm{u-\vH}_{HP} \nonumber\\*
&\leq \norm{u-\Pi u}_{HP}^2 \nonumber\\*
&\leq \sum_{\ELE\in\MESH} \left( \norm{\nabla(u-\Pi u)}_{L^2(\ELE)}^2+2 \IP\norm{(u-\Pi u)\vert_{\ELE}}^2_{L^2(\partial\ELE)}\right) \nonumber \\*
&\leq C\sum_{\ELE\in\MESH} \frac{H_{\ELE}^{2S_{\ELE}-2}}{P_{\ELE}^{2L_{\ELE}-2}}\left(1+\frac{H_{\ELE}}{P_{\ELE}}\max_{F\subset\partial\ELE}\left.\IP\right\vert_F\right)\norm{u}_{H^{L_{\ELE}}(\ELE)}^2.
\end{align}
Employing integration by parts, recalling the fact that $u$ is the analytical solution to \eqref{eqn:quasilinear_eqn}, $\mu\in C^0(\overline\Omega \times [0,\infty))$, and the bound \eqref{eqn:nl_bound:upper}, we deduce that 
\begin{align*}
&\abs{\AV(u,\wH)-\FH(\wH)} \\
&\qquad\begin{multlined}= \Bigg\vert \sum_{\ELE\in\MESH} \int_{\partial\ELE} \mu(\abs{\nabla u})\nabla u\cdot\vect{n}_{\ELE} \wH \ds \\\qquad- \sum_{F\in\FACE}\int_F \avg{\mu(\abs{\LProj(\nabla_h u)})\LProj(\nabla_h u)}\cdot\jump{\wH}\ds \Bigg\vert\end{multlined} \\
&\qquad= \abs*{ \sum_{F\in\FACE}\int_F \avg{\mu(\abs{\nabla_h u})\nabla_h u - \mu(\abs{\LProj(\nabla_h u)})\LProj(\nabla_h u)}\cdot\jump{\wH}\ds } \\
&\qquad\leq C_1 \left( \sum_{F\in\FACE}\int_F \IP^{-1}\avg{\abs{\nabla_h u - \LProj(\nabla_h u)}}\ds\right)^{\nicefrac12}\norm{\wH}_{HP}.
\end{align*}
By adding and subtracting $\Pi (\nabla_h u)$ gives
\begin{multline*}
\sum_{F\in\FACE}\int_F \IP^{-1}\avg{\abs{\nabla_h u - \LProj(\nabla_h u)}}^2\ds \\ 
 \leq2 \sum_{\ELE\in\MESH} \left(\max_{F\subset\partial\ELE}  \left.\IP^{-1}\right\vert_{F}\right)\left(\norm{\nabla u - \Pi (\nabla u)}_{L^2(\partial\ELE)}^2 + \norm{\LProj(\Pi (\nabla u) - \nabla u)}_{L^2(\partial\ELE)}^2 \right).
\end{multline*}
Applying the inverse inequality from \lemmaref{lemma:inverse_ineq} to the second term, followed by the approximation bounds stated in Corollary~\ref{corollary:coarse:approximation}, together with the $L^2$-stability of $\LProj$, we get
\begin{multline*}
\sum_{F\in\FACE}\int_F \IP^{-1}\avg{\abs{\nabla_h u - \LProj(\nabla_h u)}}^2\ds
\\ \leq  2 \sum_{\ELE\in\MESH} \left(\max_{F\subset\partial\ELE} \left.\IP^{-1}\right\vert_{F}\right)\left(C_{I,2}^\prime \frac{H_{\ELE}^{2S_{\ELE}-3}}{P_{\ELE}^{2L_{\ELE}-3}} + C_{I,1}^\prime C_sC_{\mathrm{INV}}\, d \frac{H_{\ELE}^{2S_{\ELE}-3}}{P_{\ELE}^{2L_{\ELE}-4}}\right)\norm{u}_{H^{L_{\ELE}}(\ELE)}^2.
\end{multline*}
Combining \eqref{theorem:coarse_priori:proof:step1} and the above bounds with \lemmaref{lemma:abstract:priori} completes the proof.
\end{proof}

\subsection{Two-grid \emph{a priori} error bound}
We now consider the derivation of an \emph{a priori} error bound for the two-grid DGFEM \eqref{eqn:tg_fem:coarse}--\eqref{eqn:tg_fem:fine}. To this end, we first recall the following \emph{a priori} error bound for the standard DGFEM approximation \eqref{eqn:standard_fem} of the quasilinear problem \eqref{eqn:quasilinear_eqn}.
\begin{lemma}
\label{lemma:fine_priori}
Assuming that $u\in C^1(\Omega)$ and $u\vert_{\ele}\in H^{l_{\ele}}(\ele)$, $l_{\ele} \geq 2$, for $\ele\in\mesh$ then the solution $\uh\in\dgspace$ of \eqref{eqn:standard_fem} satisfies the error bound
\[
\norm{u-\uh}_{hp}^2 \leq C_4 \sum_{\ele\in\mesh} \frac{h_{\ele}^{2s_{\ele}-2}}{p_{\ele}^{2l_{\ele}-3}} \norm{u}_{H^{l_{\ele}}(\ele)}^2,
\]
with $1\leq s_{\ele} \leq \min(p_{\ele}+1,l_{\ele})$, $p_{\ele}\geq 1$, for $\ele\in\mesh$, and $C_4$ is a positive constant independent of $u$, $h$, and $\vect p$, but depends on the constants $m_{\mu}$, $M_{\mu}$, $C_1$, and $C_2$ from the monotonicity properties of $\mu(\cdot)$.
\end{lemma}
\begin{proof}
See \cite{Houston_QL_priori}.
\end{proof}

Employing \theoremref{theorem:coarse_priori} and \lemmaref{lemma:fine_priori}, we now deduce the following error bound for the two-grid approximation defined in \eqref{eqn:tg_fem:coarse}--\eqref{eqn:tg_fem:fine}.
\begin{theorem}
\label{theorem:tg_priori}
Let $\MESH$ be a coarse mesh, constructed by agglomerating elements from a shape-regular fine mesh $\mesh$, satisfying Assumptions \ref{assumption:coarse:shape_regular} and \ref{assumption:coarse:covering}, with $\MESH^{\sharp}=\{\COVERELE\}$ an associated covering of $\MESH$ consisting of $d$-simplices; cf. \defref{def:covering}. If the analytical solution $u\in H^1(\Omega)$ to \eqref{eqn:quasilinear_eqn} satisfies $u\vert_{\ele}\in H^{l_{\ele}}(\ele)$, $l_{\ele} \geq 2$ and $u\vert_{\ELE}\in H^{L_{\ELE}}(\ELE)$, $L_{\ELE} \geq \nicefrac32$, for $\ELE\in\MESH$, such that $\mathfrak{E}u\vert_{\COVERELE}\in H^{L_{\ELE}}(\COVERELE)$, where $\COVERELE\in\MESH^{\sharp}$ with $\ELE\subset\COVERELE$; then,
the solution $\ut\in\dgspace$ of \eqref{eqn:tg_fem:fine} satisfies the error bounds
\begin{align*}
\norm{\uh-\ut}_{hp}^2 & \leq C_4C_5 \sum_{\ele\in\mesh} \frac{h_{\ele}^{2s_{\ele}-2}}{p_{\ele}^{2l_{\ele}-3}} \norm{u}_{H^{l_{\ele}}(\ele)}^2 \\*
    &\qquad+ C_3C_5 \sum_{\ELE\in\MESH} \frac{H_{\ELE}^{2S_{\ELE}-2}}{P_{\ELE}^{2L_{\ELE}-2}} (1+\mathcal{G}_{\ELE}(H_{\ELE},P_{\ELE})) \norm{u}_{H^{L_{\ELE}}(\ELE)}^2,\\
\norm{u-\ut}_{hp}^2 &\leq (1+C_5)C_4 \sum_{\ele\in\mesh} \frac{h_{\ele}^{2s_{\ele}-2}}{p_{\ele}^{2l_{\ele}-3}} \norm{u}_{H^{l_{\ele}}(\ele)}^2 \\*
&\qquad+ C_3 C_5 \sum_{\ELE\in\MESH} \frac{H_{\ELE}^{2S_{\ELE}-2}}{P_{\ELE}^{2L_{\ELE}-2}} (1+\mathcal{G}_{\ELE}(H_{\ELE},P_{\ELE})) \norm{u}_{H^{L_{\ELE}}(\ELE)}^2,
\end{align*}
where $S_{\ELE} = \min(P_{\ELE}+1,L_{\ELE})$, for $\ELE\in\MESH$, $s_{\ele} = \min(p_{\ele}+1,l_{\ele})$, for $\ele\in\mesh$,
\[
    \mathcal{G}_{\ELE}(H_{\ELE},P_{\ELE})\coloneqq \frac{P_{\ELE}+P_{\ELE}^2}{H_{\ELE}}\max_{F\subset\partial\ELE} \left.\IP^{-1}\right\vert_F + \frac{H_{\ELE}}{P_{\ELE}}\max_{F\subset\partial\ELE} \left.\IP\right\vert_F
\]
and $C_5$ is a positive constant independent of $u$, $h$, $H$, $\vect p$, and $\vect P$, but depends on the constants $m_{\mu}$, $M_{\mu}$, $C_1$, and $C_2$ from the monotonicity properties of $\mu(\cdot)$.
\end{theorem}
\begin{remark}
Assuming that $\vect{P}$ is of local bounded variation, we note that due to the definition of $\IP$, and the fact that $P_{\ELE}\geq 1$ for all $\ELE\in\MESH$, we have that $\mathcal{G}_{\ELE}(H_{\ELE},P_{\ELE}) \leq C P_{\ELE}$, for some positive constant $C$, independent of mesh size and polynomial degree. Therefore, we note that the terms in the second error bound have the same order for the polynomial degree and mesh size in both the coarse and fine mesh discretization parameters, which is analogous to the case when non-agglomerated coarse meshes, i.e. coarse meshes consisting of standard element types, are employed; cf. \cite{CongreveTG}
\end{remark}
\begin{proof}
By application of the triangle inequality, we get
\[
\norm{u-\ut}_{hp} \leq \norm{u-\uh}_{hp} + \norm{\ut-\uh}_{hp};
\]
hence, once the first bound stated in Theorem~\ref{theorem:tg_priori} has been established, then together with Lemma~\ref{lemma:fine_priori}, the second bound follows immediately.

For the first bound, we follow the proof in \cite[Theorem 3.1]{CongreveTG}. From the definition of the standard DGFEM formulation \eqref{eqn:standard_fem} and the fine grid approximation \eqref{eqn:tg_fem:fine} we have that
\[
    \Ah(\uh;\uh,\vh) = \Ah(\uH;\ut,\vh) \qquad \text{for all }\vh\in\dgspace.
\]
Therefore, letting $\phi=\ut-\uh\in\dgspace$, from  \lemmaref{lemma:coercivity}, we deduce that
\begin{align*}
C_c\norm{\phi}_{hp}^2 &\leq \Ah(\uh; \uh, \phi) - \Ah(\uH; \uh, \phi) \\*
&\leq \sum_{\ele\in\mesh} \int_{\ele} \abs{(\mu(\abs{\nabla\uh})-\mu(\abs{\nabla\uH}))\nabla\uh}\abs{\nabla\phi}\dx \\*
    &\qquad +\sum_{F\in\face}\int_F \avg{\abs{(\mu(\abs{\nabla_h\uh})-\mu(\abs{\nabla_h\uH}))\nabla_h\uh}}\abs{\jump{\phi}}\ds .
\end{align*}
Adding and subtracting $\mu(\abs{\nabla \uH})\nabla\uH$ to both terms on the right-hand side, then applying the triangle inequality, together with \eqref{eqn:monotonicity} and \eqref{eqn:nl_bound:upper} gives
\begin{align*}
&\norm{\phi}_{hp}^2 \\
&\leq \frac{C_1+M_\mu}{C_c}\left(\sum_{\ele\in\mesh} \int_{\ele} \abs{\nabla(\uh-\uH)}\abs{\nabla\phi}\dx +\sum_{F\in\face}\int_F \avg{\abs{\nabla_h(\uh-\uH)}}\abs{\jump{\phi}}\ds\right) \\
&\begin{multlined}
\leq \frac{C_1+M_\mu}{C_c} \left( \left(\sum_{\ele\in\mesh}\norm{\nabla(u-\uh)}_{L^2(\ele)}^2+\sum_{F\in\face}\int_F \ip^{-1}\abs{\avg{\abs{\nabla_h(u-\uh)}}}^2\ds\right)^{\nicefrac12} \right. \\
+\left.\left(\sum_{\ELE\in\MESH}\norm{\nabla(u-\uH)}_{L^2(\ELE)}^2+\sum_{F\in\face}\int_F \ip^{-1}\abs{\avg{\abs{\nabla_h(u-\uH)}}}^2\ds\right)^{\nicefrac12}  \right)\norm{\phi}_{hp}.
\end{multlined}
\end{align*}
Therefore, applying a standard $hp$--version trace inequality, along with local bounded variation of the mesh parameters, we deduce that
\begin{multline*}
\norm{\ut-\uh}_{hp} \\\leq \frac{\sqrt{C_5}}{2} \left(\left( \sum_{\ele\in\mesh}\norm{\nabla(u-\uh)}_{L^2(\ele)}^2 \right)^{\nicefrac12} + \left( \sum_{\ELE\in\MESH}\norm{\nabla(u-\uH)}_{L^2(\ELE)}^2 \right)^{\nicefrac12} \right),
\end{multline*}
where $C_5$ is a positive constant independent of $u$, $h$, $H$, $\vect p$, and $\vect P$, but depends on the constants $m_{\mu}$, $M_{\mu}$, $C_1$, and $C_2$.
Applying \lemmaref{lemma:fine_priori} and \theoremref{theorem:coarse_priori} completes the proof.
\end{proof}
A numerical example validating these bounds can be found in our conference article \cite{CongreveGAMM2019}.

\section{\emph{A posteriori} error estimation and two-grid \texorpdfstring{$hp$}{hp}-adaptive refinement}
\label{sec:apost_and_adaptivity}

We note that the existing \emph{a posteriori} error bound \cite[Theorem 3.2]{CongreveTG} still holds for an agglomerated coarse mesh, as the only requirement on the coarse mesh is the fact that $\DGSPACE\subset\dgspace$, which is still true in the current setting. For completeness we reproduce the error bound here.

\begin{theorem}
\label{theorem:posteriori}
Let $u\in H^1(\Omega)$ be the analytical solution of~\eqref{eqn:quasilinear_eqn}, $\uH\in\DGSPACE$ the numerical approximation obtained from~\eqref{eqn:tg_fem:coarse}, and $\ut\in\dgspace$ the numerical approximation computed from~\eqref{eqn:tg_fem:fine}; then, the following $hp$-\emph{a posteriori} error bound holds
\[
    \norm{u-\ut}_{hp} \leq C_6 \left( \sum_{\ele\in\mesh} (\eta_{\ele}^2 + \xi_{\ele}^2) + \sum_{\ele\in\mesh} \norm{f-\Pi_{L^2}f}_{L^2(\ele)}^2\right)^{\nicefrac12},    
\]
with a constant $C_6>0$, which is independent of $\vect{h}$, $\vect{H}$, $\vect{p}$, and $\vect{P}$. Here, $\Pi_{L^2}$ is the $L^2$-projection onto the fine grid finite element space $\dgspace$, the local \emph{fine grid error indicators} $\eta_{\ele}$ are defined, for all $\ele\in\mesh$, by
\begin{multline}
\label{eqn:error_indicators:fine}
\eta_{\ele}^2 = h_{\ele}^2p_{\ele}^{-2}\norm{\Pi_{L^2}f+\nabla\cdot(\mu(\abs{\nabla \uH})\nabla \ut)}_{L^2(\ele)}^2 \\
    +h_{\ele} p_{\ele}^{-1} \norm{\jump{\mu(\abs{\nabla \uH})\nabla \ut}}_{L^2(\partial \ele\setminus\Gamma)}^2
    + \gamma^2 h_{\ele}^{-1}p_{\ele}^3 \norm{\jump{\ut}}_{L^2(\partial \ele)}^2,
\end{multline}
and the local \emph{two-grid error indicators} $\xi_{\ele}$ are defined, for all $\ele\in\mesh$, by
\begin{equation}
\label{eqn:error_indicators:twogrid}
\xi_{\ele}^2 = \norm{(\mu(\abs{\nabla \uH})-\mu(\abs{\nabla \ut}))\nabla \ut}_{L^2(\ele)}^2.
\end{equation}
\end{theorem}

For the two-grid DGFEM discretization defined by \eqref{eqn:tg_fem:coarse}--\eqref{eqn:tg_fem:fine} it is necessary to refine both fine and coarse meshes, together with their corresponding polynomial degree vectors, in order to decrease the error between $u$ and $\ut$ with respect to the energy norm $\norm{\cdot}_{hp}$. We note that, from \theoremref{theorem:posteriori}, we have, for each fine element $\ele\in\mesh$, a local error indicator $\eta_{\ele}$ and a local two-grid error indicator $\xi_{\ele}$. The local error indicator $\eta_{\ele}$ is similar to the one which arises within the analysis of the standard DGFEM discretization and, hence, represents the error arising from the linear fine grid solve \eqref{eqn:tg_fem:fine}; whereas, the local two-grid error indicator $\xi_{\ele}$ represents the error stemming from the approximation of the nonlinear coefficient $\mu(\abs{\nabla \uh})$ on the fine mesh by the nonlinear coefficient evaluated with respect to the coarse grid solution $\uH$. To this end, we can consider a modified version of the two-grid mesh refinement algorithm \cite[Algorithm 1 \& Algorithm 2]{CongreveNonNewtonTG} to allow for a coarse mesh  consisting of agglomerated fine mesh elements.
\begin{algo}
\label{algo:refinement}
The fine and coarse finite element spaces $\dgspace$ and $\DGSPACE$ are refined as follows.
\begin{enumerate}
\item Initial Step: Select an initial fine mesh $\mesh$ and initial fine mesh polynomial degree distribution $\vect{p}$. Create a coarse mesh $\MESH$ by element agglomeration/graph partitioning (e.g., by METIS \cite{Karypis1999}) and assign a polynomial degree distribution $\vect{P}$, such that $\DGSPACE\subseteq\dgspace$.
\item Solve~\eqref{eqn:tg_fem:coarse}--\eqref{eqn:tg_fem:fine} to determine $\uH$ and $\ut$, respectively. \label{algo:refinement:initial_step}
\item Select elements for refinement based on the local fine grid error indicators $\eta_{\ele}$ and the local two-grid error indicators $\xi_{\ele}$, from~\eqref{eqn:error_indicators:fine} and~\eqref{eqn:error_indicators:twogrid}, respectively:
    \begin{enumerate}
    \item Determine the set $\mathfrak{R}(\mesh)\subseteq\mesh$ of \emph{potential} elements to refine based on $\sqrt{\eta_{\ele}^2+\xi_{\ele}^2}$ using a standard refinement strategy, e.g., the fixed fraction strategy.
    \item For all elements selected for refinement decided whether to perform refinement on the fine or coarse mesh. For all $\ele\in\mathfrak{R}(\mesh)$:
        \begin{itemize}
        \item if $\lambda_F\xi_{\ele}\leq\eta_{\ele}$ refine the fine element $\ele$, and
        \item if $\lambda_C\eta_{\ele}\leq\xi_{\ele}$ refine the coarse element $\ELE\in\MESH$, where $\ele\in\mesh(\ELE)$.
        \end{itemize}
    \end{enumerate}
\item Perform $hp$-refinement on the fine mesh $\mesh$ using a standard refinement method; see, for example, \cite{HoustonSuli, Mitchell2011, Mitchell2014, Wihler2011}.
\item For elements marked for refinement in the coarse mesh $\MESH$ determine whether to perform $h$- or $p$-refinement; see, for example, \cite{HoustonSuli, Mitchell2011, Mitchell2014, Wihler2011}.
\item Perform mesh smoothing of the fine mesh to ensure that for any coarse element $\ELE\in\mesh$ marked for $h$-refinement that $\mesh(\ELE)$ contains at least $2^d$ fine mesh elements. Also, for any coarse mesh element $\ELE\in\MESH$ marked for $p$-refinement, if there exists a fine mesh element $\ele\in\mesh(\ELE)$ such that $p_{\ele}=P_{\ELE}$ do not perform $p$-refinement on $\ELE$.
\item Perform $hp$-refinement on the coarse mesh.
\end{enumerate}
Here, $\lambda_F, \lambda_C \in (0,\infty)$ are steering parameters selected such that $\lambda_F\lambda_C \leq 1$.
\end{algo}
\begin{remark}
For the purposes of the numerical experiments in the following section the initial coarse mesh, in Step~\ref{algo:refinement:initial_step} above, is selected by agglomerating the fine mesh into $\lceil\nicefrac{N}{2^d}\rceil$ coarse elements, where $N$ is the number of fine mesh elements, and the initial polynomial degrees for all fine and coarse elements are set to the same polynomial degree; i.e., for a polynomial degree $p$ we set $p_{\ele}=p$ for all $\ele\in\mesh$ and $P_{\ELE}=p$ for all $\ELE\in\MESH$.
\end{remark}

In order to perform refinement on the coarse element we need an algorithm for $h$-refinement of agglomerated elements. The first potential algorithm is a na\"ive approach based on simply agglomerating the sub-elements on a coarse element marked for refinement into smaller elements; cf. \cite{Collis2016}.
\begin{algo}[Na\"ive (Unweighted) Coarse Refinement]
\label{algo:naive}
For each $\ELE\in\MESH$ marked for refinement partition the sub-patch $\mesh(\ELE)$ into $2^d$ elements using graph partitioning (e.g., by METIS \cite{Karypis1999}).
\end{algo}
The standard graph partitioning algorithm subdivides the elements into partitions containing a roughly equal number of elements. However, given that we have information on the likely local error size for each fine mesh element, it should be possible to refine the coarse mesh elements to equidistribute the magnitude of the error indicators to the new elements. To this end, we note that METIS provides a means of performing graph partitioning based on allocating weights for each vertex, cf. \cite{Karypis1998}. Exploiting this procedure, we propose the following alternative algorithm.
\begin{algo}[Weighted Coarse Refinement]
\label{algo:errors}
For each coarse element $\ELE\in\MESH$ marked for refinement, we allocate some weight $\omega_{\ele}\in\real$ to its fine sub-elements $\ele\in\mesh(\ELE)$ based on the error indicators; i.e, we set
\[
    \omega_{\ele}=\eta_{\ele}^2+\xi_{\ele}^2.
\]
We then refine the coarse element $\ELE\in\MESH$ as follows:
\begin{enumerate}
\item Construct an adjacency graph for $\mesh(\ELE)$, with a vertex $\mathcal{N}_{\ele}$ for each element $\ele\in\mesh(\ELE)$, and an edge $\mathcal{E}_F$ connecting the vertices $\mathcal{N}_{\ele}, \mathcal{N}_{\ele'}$ of each pair of elements $\ele,\ele'$ which share a common face $F\in\face[I](\ELE)\coloneqq\{ F\in\face[I] : F=\partial\ele\cap\partial\ele', \ele,\ele'\in\mesh(\ELE) \}$.
\item Assign the weights $\omega_{\ele}$ to the vertex $\mathcal{N}_{\ele}$ for each element $\ele\in\mesh(\ELE)$.
\item Perform graph partitioning on the adjacency graph to partition the graph into $2^d$ sub-graphs such that the sum of the weights in each sub-graph is (roughly) equal.
\item Construct the new refined elements from the sub-graphs.
\end{enumerate}
\end{algo}

\begin{figure}[t]
    \configfigure
    \subfloat[$\mesh$ (gray) and $\MESH$ (black)]{\label{fig:coarse_refine:step1}\includegraphics[width=0.35\textwidth]{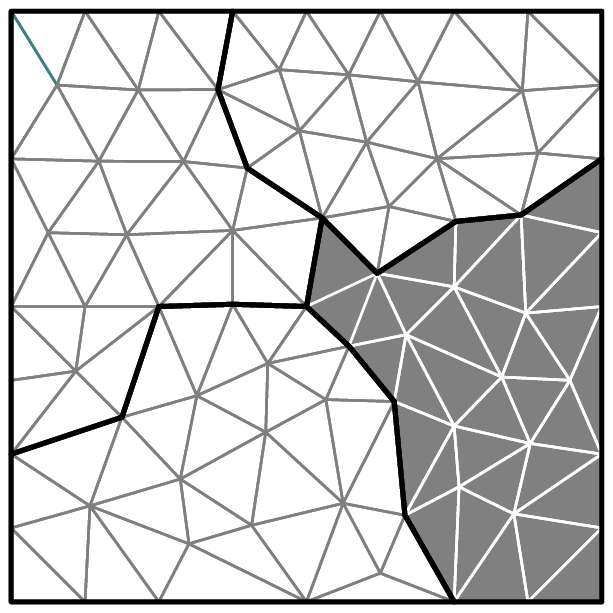}}
    \subfloat[$\mesh(\ELE)$, $\ELE\in\MESH$]{\label{fig:coarse_refine:step2}\includegraphics[width=0.35\textwidth]{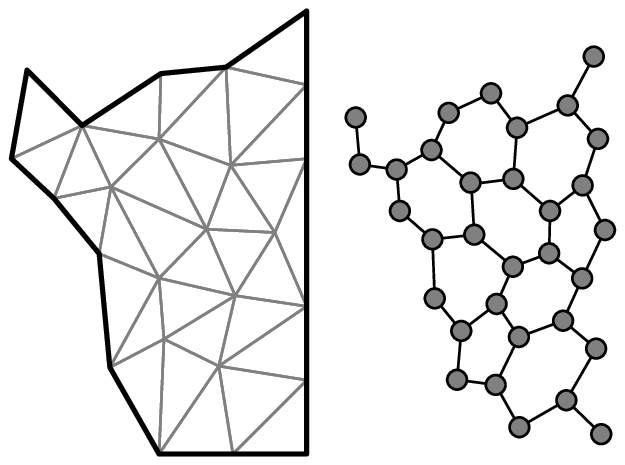}} \\
    \subfloat[\algoref{algo:naive}]{\label{fig:coarse_refine:step3}\includegraphics[width=0.35\textwidth]{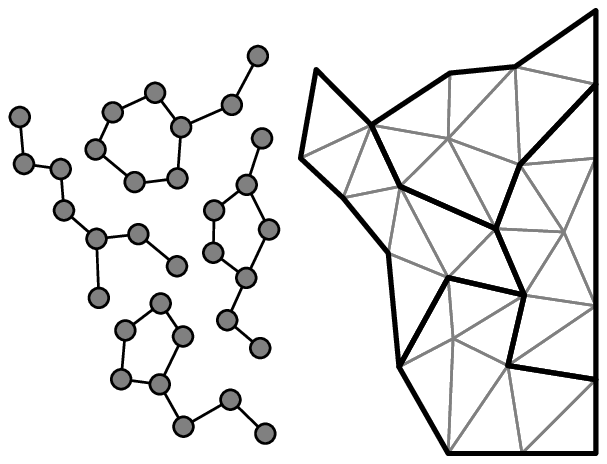}}
    \subfloat[\algoref{algo:errors}]{\label{fig:coarse_refine:step4}\includegraphics[width=0.35\textwidth]{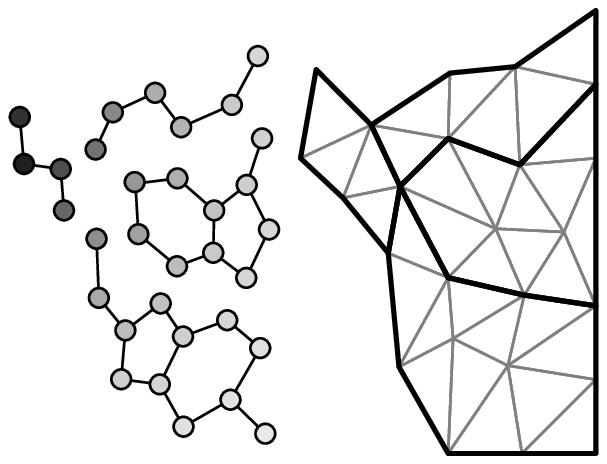}}
    \caption{Coarse element refinement example. \protect\subref{fig:coarse_refine:step1} Fine mesh (gray) agglomerated into 4 coarse elements (black), with shaded element $\ELE\in\MESH$ marked for $h$-refinement; \protect\subref{fig:coarse_refine:step2} The element $\ELE\in\MESH$ marked for $h$-refinement, the constituent fine elements $\mesh(\ELE)$, and the adjacency graph for these fine elements. \protect\subref{fig:coarse_refine:step3} \algoref{algo:naive}. Graph partitioning of the adjacency graph, and the resulting coarse element refinement. \protect\subref{fig:coarse_refine:step3} \algoref{algo:errors}. Graph partitioning of the adjacency graph with vertex weights $\omega_\ele=\eta_{\ele}^2+\xi_{\ele}^2$, $\ele\in\mesh(\ELE)$, denoted by vertex color (black $=0.5$, white $=0$), and the resulting coarse element refinement}
    \label{fig:coarse_refine}
\end{figure}

Note, this algorithm is performed after fine mesh refinement; therefore, we divide the error indicators $\eta_{\ele}$ and $\xi_{\ele}$ of a refined fine mesh element $\ele\in\mesh$ between its new elements. To that end, we change the fine mesh refinement algorithm to compute new \emph{effective} error indicators:
\begin{algo}
We calculate the \emph{effective} error indicators $\eta_{\ele}'$ and $\xi_{\ele}'$ on the fine mesh after mesh refinement from $\eta_{\ele}$ and $\xi_{\ele}$ as follows.
\begin{algorithmic}
\ForAll{$\ele\in\mesh$}
    \If{$\ele$ is marked for $h$-refinement}
        \State Perform $h$-refinement, dividing the element into $N$ children: $\ele_{1},\dots,\ele_{N}$
        \State $\eta_{\ele_{i}}' \gets \frac{\eta_{\ele}}{\sqrt{N}}$, for $i=1,\dots,N$
        \State $\xi_{\ele_{i}}' \gets \frac{\xi_{\ele}}{\sqrt{N}}$, for $i=1,\dots,N$
    \Else
        \State $\eta_{\ele}' \gets \eta_{\ele}$
        \State $\xi_{\ele}' \gets \xi_{\ele}$
    \EndIf
\EndFor
\end{algorithmic}
\end{algo}

\figref{fig:coarse_refine} demonstrates Algorithms~\ref{algo:naive} \&~\ref{algo:errors} for an example coarse element refinement. \subfigref{fig:coarse_refine}{step1} shows an example fine mesh $\mesh$ with corresponding coarse mesh $\MESH$ constructed by agglomerating the fine mesh into four elements, and highlights one coarse element $\ELE\in\MESH$ for $h$-refinement. In \subfigref{fig:coarse_refine}{step2} we take the constituent fine elements $\mesh(\ELE)$ of the element $\ELE\in\MESH$ marked for $h$-refinement and create the matching adjacency graph for these elements. Figs.~\ref{fig:coarse_refine}\subref{fig:coarse_refine:step3} \&~\ref{fig:coarse_refine}\subref{fig:coarse_refine:step4} show how \algoref{algo:naive} and \algoref{algo:errors}, respectively, partition the adjacency graph into $2^d=4$ sub-graphs, and the matching coarse element refinement.

\section{Numerical experiments}\label{section:numerics}

In this section we perform a series of numerical experiments to demonstrate the performance of the \emph{a posteriori} error bound stated in \theoremref{theorem:posteriori}, the $hp$-adaptive mesh refinement strategy outlined in \algoref{algo:refinement}, and the coarse mesh refinement strategies presented in Algorithms~\ref{algo:naive} \&~\ref{algo:errors}, using both $h$- and $hp$-adaptive mesh refinement. We set the interior penalty parameters $\gamma_{hp}$ and $\gamma_{HP}$ in \eqref{eqn:ip:fine} and \eqref{eqn:ip:coarse}, respectively, equal to $10$. For the two steering parameters from \algoref{algo:refinement} we set $\lambda_C=\nicefrac12$ and $\lambda_F=1$. The nonlinear equations are solved by employing a damped Newton iteration \cite[Sect. 14.4]{Ortega}. The solution of the resulting set of linear equations, from either the fine mesh or at each step of the iterative nonlinear solver, is computed using either the direct MUMPS solver \cite{Amestoy}, for two-dimensional problems or an ILU preconditioned GMRES algorithm \cite{Saad}, for the three-dimensional problems presented here. We also calculate effectivity indices by dividing the error bound stated in \theoremref{theorem:posteriori}, with the constant $C_6$ set to 1, by the error computed in the DGFEM energy norm.

For comparison purposes, for each example presented below, in addition to the $h$- and $hp$-version adaptive two-grid algorithms presented in Section~\ref{sec:apost_and_adaptivity}, we also perform $h$- and $hp$-adaptive refinement using the standard DGFEM formulation \eqref{eqn:standard_fem}.

\subsection{Example 1: Smooth analytical solution}
In this example, we let $\Omega$ be the unit square $(0,1)^2\subset\real^2$, define the nonlinear coefficient by
\begin{equation}
\label{eqn:simple_mu}
\mu(\abs{\nabla u}) = 2 + \frac{1}{1+\abs{\nabla u}},
\end{equation}
and select the forcing function $f$ such that the analytical solution to \eqref{eqn:quasilinear_eqn} is given by
\[
u(x,y) = x(1-x)y(1-y)(1-2y)\mathrm{e}^{-20(2x-1)^2}.
\]

\begin{figure}[t]
    \configfigure
    \subfloat[]{\label{fig:square:err}\includegraphics[width=0.45\textwidth]{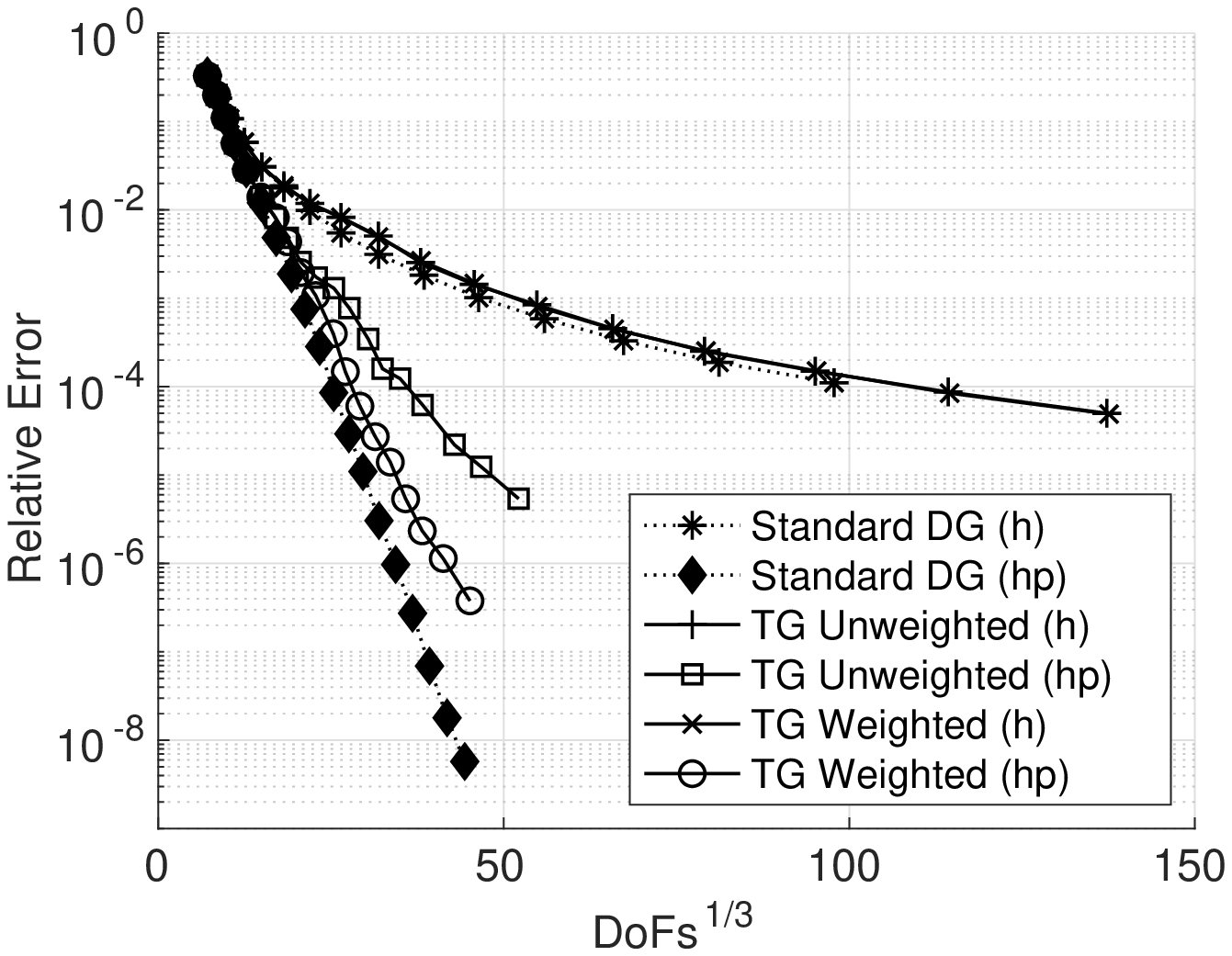}}
    \subfloat[]{\label{fig:square:eff}\includegraphics[width=0.45\textwidth]{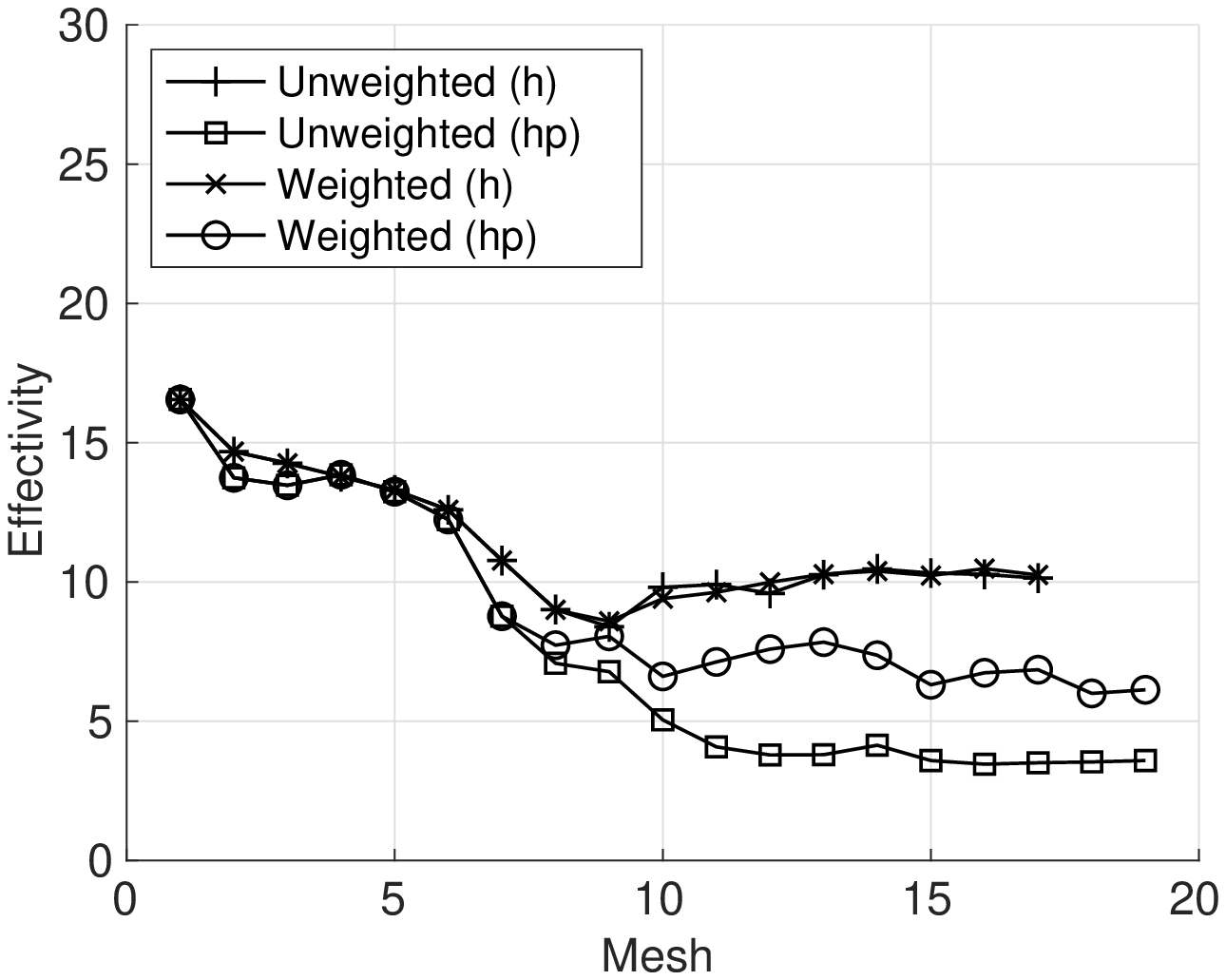}} \\
    \subfloat[]{\label{fig:square:dofs}\includegraphics[width=0.45\textwidth]{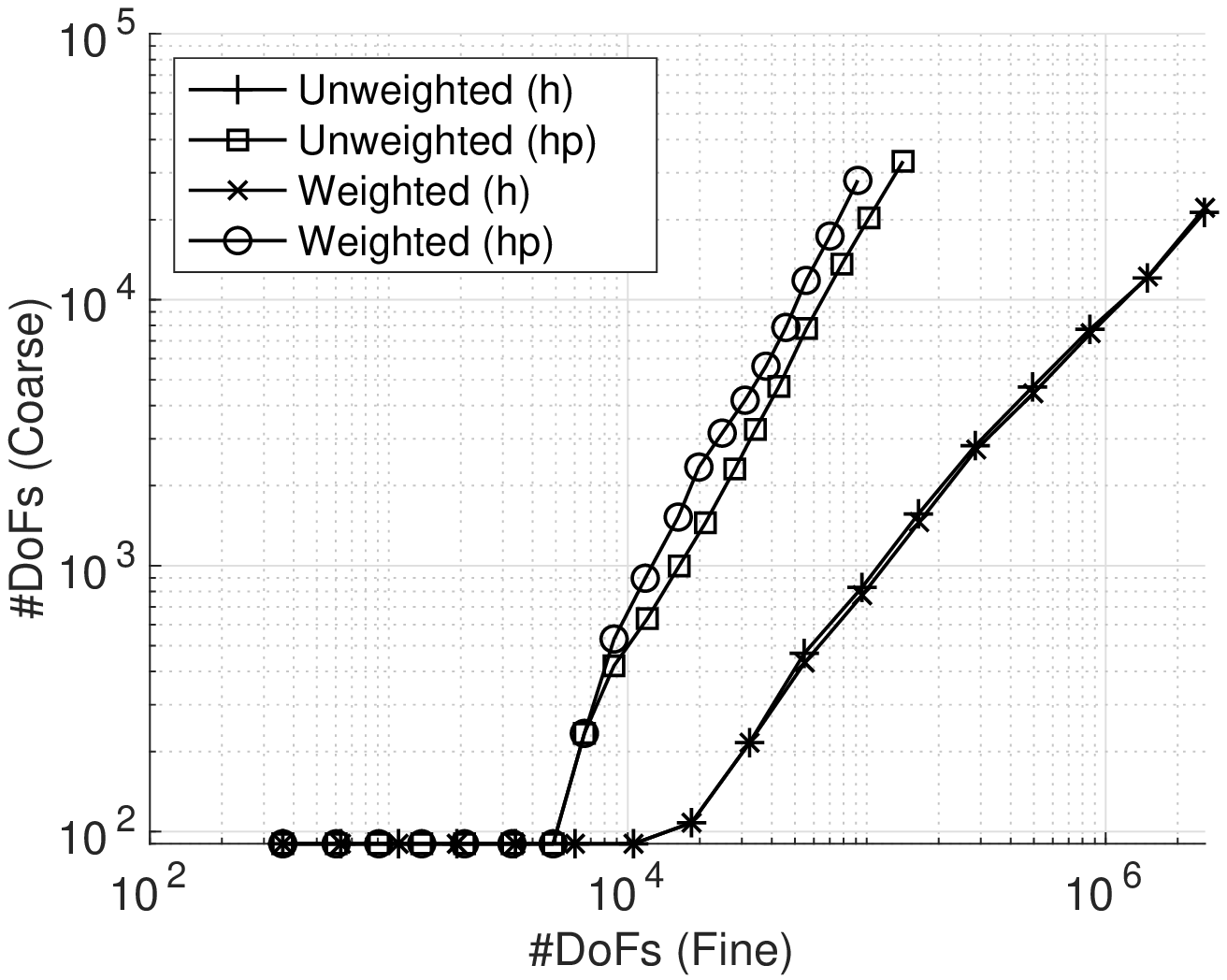}}
    \caption{Example 1. \protect\subref{fig:square:err}~Error in the DG norm with respect to the number of degrees of freedom for the standard method and the two-grid methods, using weighted and unweighted coarse mesh refinement, with $h$- and $hp$-refinement; \protect\subref{fig:square:eff}~Effectivity indices for the two-grid method with $h$- and $hp$-refinement; \protect\subref{fig:square:dofs}~Comparison of the number of degrees of freedom on the coarse and fine mesh for the two-grid methods}
    \label{fig:square}
\end{figure}

In \subfigref{fig:square}{err} we present the relative error measured in the energy norm versus the third root of the number of degrees of freedom (in the fine finite element space $\dgspace$) for the standard DGFEM formulation \eqref{eqn:standard_fem}, together with the corresponding quantities computed based on employing the two-grid DGFEM formulation \eqref{eqn:tg_fem:coarse}--\eqref{eqn:tg_fem:fine} using both \algoref{algo:naive} (TG Unweighted) and \algoref{algo:errors} (TG Weighted) for the coarse mesh refinement. Here, we perform both $h$- and $hp$-adaptive refinement for all methods (independently). We observe that, for the problem at hand, when $h$-refinement is employed the two two-grid methods lead to a slight increase in the error measured in the DGFEM norm, relative to the standard DGFEM formulation, in the sense that for a fixed number of degrees of freedom the latter is slightly superior. In the $hp$-refinement setting, we note that exponential convergence is observed for all three methods as the underlying finite element space is enriched, although we notice that when unweighted coarse mesh refinement procedure, cf. \algoref{algo:naive}, is employed within the two-grid method, then the norm of the error has a noticeably slower rate of convergence. In \subfigref{fig:square}{eff}, we display the effectivity indices calculated by dividing the error bound by the true error measured in the energy norm for each of DGFEMs and refinement strategies employed. We note that initially the effectivity indices drop before roughly stabilizing to a constant, thus indicating that the \emph{a posteriori} error bound overestimates the true error by a roughly consistent amount.

Although \subfigref{fig:square}{err} suggests that the two-grid methods perform worse than the standard DGFEM, when considering the magnitude of the error measured in the DGFEM norm relative to the number of degrees of freedom employed in the fine finite element space $\dgspace$, this degradation is expected since we are only solving a linearized version of the underlying numerical scheme on $\dgspace$. However, as the coarse space $\DGSPACE$ should contain considerably fewer degrees of freedom than $\dgspace$, we expect the two-grid method to be computationally cheaper as it only solves the nonlinear problem on $\DGSPACE$. With this mind, in \subfigref{fig:square}{dofs} we compare the number of degrees of freedom in $\DGSPACE$ and $\dgspace$ for both coarse mesh refinement strategies, Algorithms~\ref{algo:naive} \&~\ref{algo:errors}, when both $h$- and $hp$-refinement are employed. As expected, the number of degrees of freedom in the coarse mesh is considerable lower compared to the fine mesh; furthermore, we notice that both the unweighted, \algoref{algo:naive}, and weighted, \algoref{algo:errors}, coarse mesh refinement algorithms result in a similar number of coarse mesh degrees of freedom compared to the fine mesh.
\begin{figure}[t]
    \configfigure
    \subfloat[$h$-refinement]{\label{fig:square:timing:h}\includegraphics[width=0.45\textwidth]{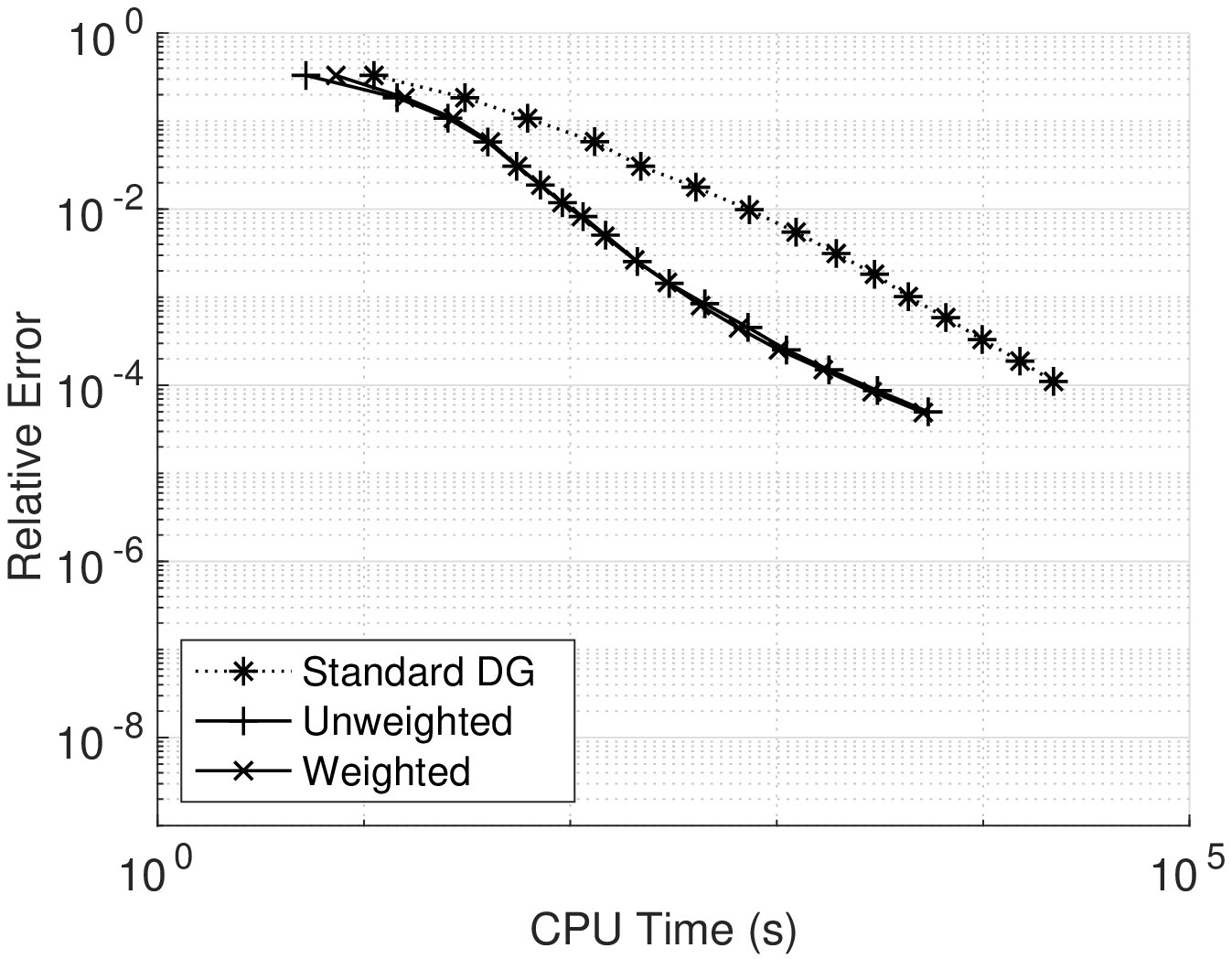}}
    \subfloat[$hp$-refinement]{\label{fig:square:timing:hp}\includegraphics[width=0.45\textwidth]{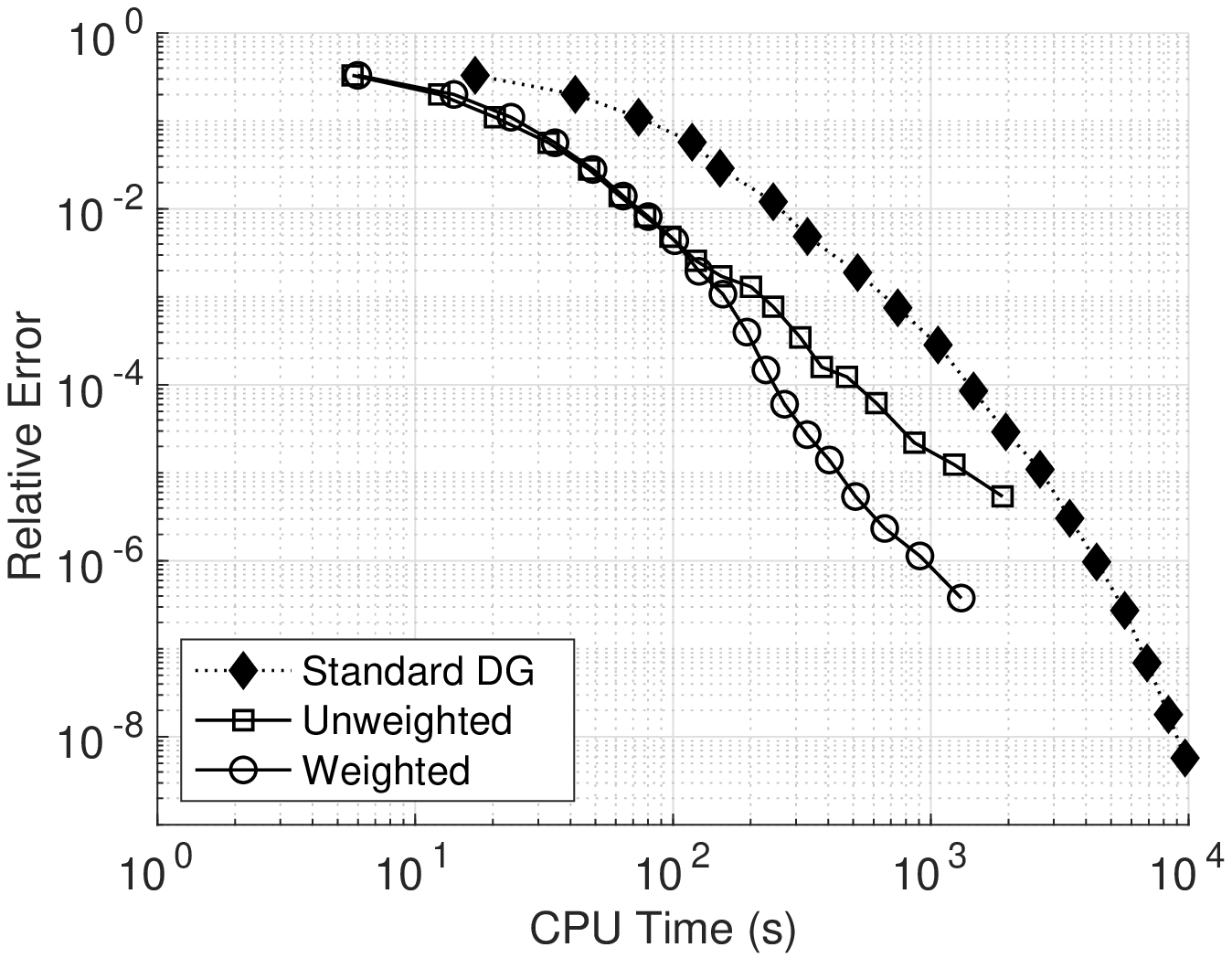}}
    \caption{Example 1. Error in the DG norm with respect to the cumulative CPU time for the standard method and the two-grid methods, using weighted and unweighted coarse mesh refinement, with \protect\subref{fig:square:timing:h}~$h$- and \protect\subref{fig:square:timing:hp}~$hp$-refinement}
    \label{fig:square:timing}
\end{figure}

To investigate this issue further in \figref{fig:square:timing}, we compare the relative error computed in the energy norm with the cumulative computation time for both the standard DGFEM and the two two-grid DGFEMs employing the different coarse mesh refinement strategies, when both $h$- and $hp$-adaptive mesh refinement is exploited. In the $h$-refinement setting the two two-grid DGFEMs lead to around an order of magnitude decrease in the error measured in the DGFEM norm, when compared to the standard DGFEM, for a given fixed computation time.
When $hp$-refinement is employed, the reduction in the error in the two-grid DGFEM compared to the standard DGFEM, for a given fixed amount of computation time, increases to roughly two orders of magnitude when the weighted coarse mesh refinement strategy, cf. \algoref{algo:errors}, is employed. However, when the unweighted coarse mesh refinement algorithm is employed within the two-grid DGFEM, cf. \algoref{algo:naive}, this improvement in the error computed in the DGFEM norm decreases as refinement progresses; this is caused by the noticeably slower rate of convergence observed in \subfigref{fig:square}{err}.
This result, along with the fact that both coarse mesh refinement algorithms result in a broadly similar number of degrees of freedom on the coarse mesh, suggest that the weighted \algoref{algo:errors} coarse mesh refinement is a superior refinement strategy in the $hp$-setting.

\begin{figure}[t]
    \configfigure
    \subfloat[Coarse ($h$-refinement)]{\label{fig:square:mesh:h:coarse}\includegraphics[width=0.45\textwidth]{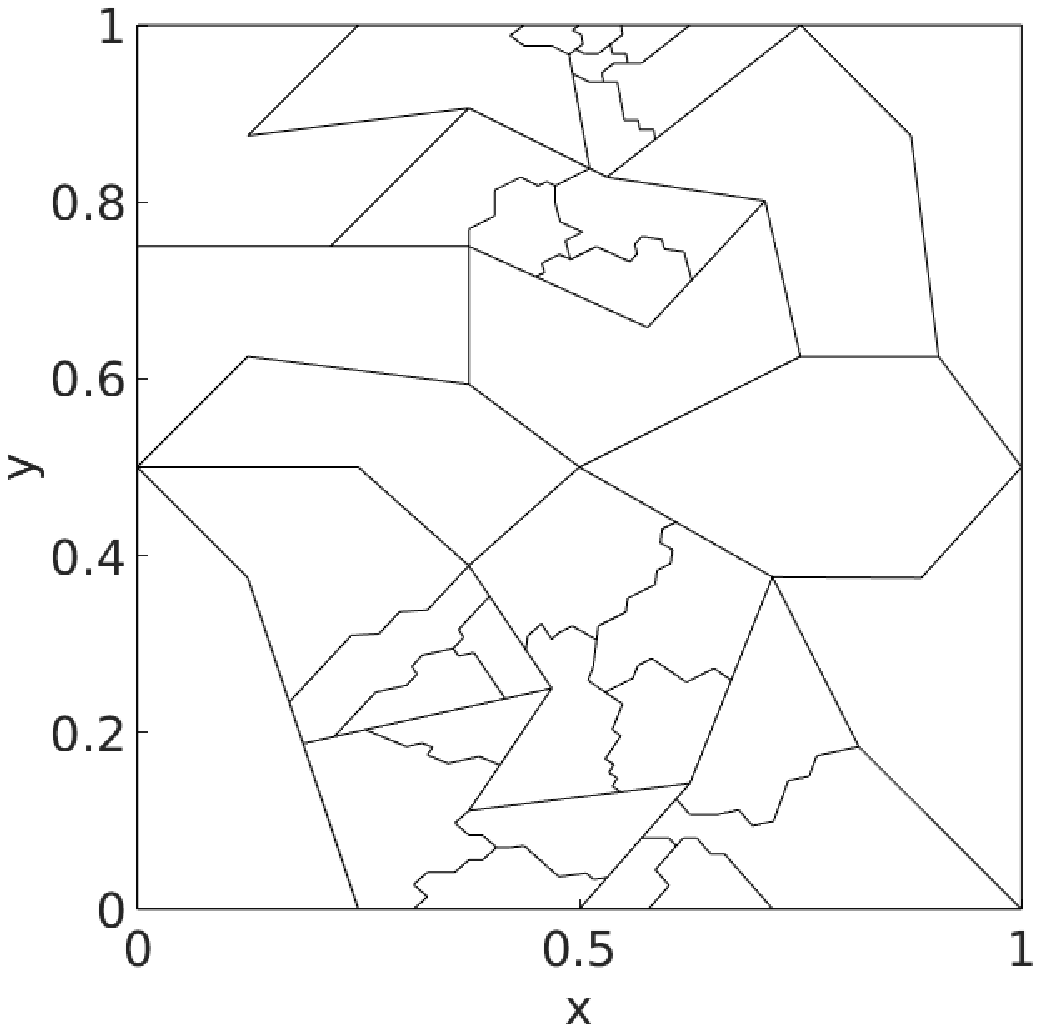}}
    \subfloat[Fine ($h$-refinement)]{\label{fig:square:mesh:h:fine}\includegraphics[width=0.45\textwidth]{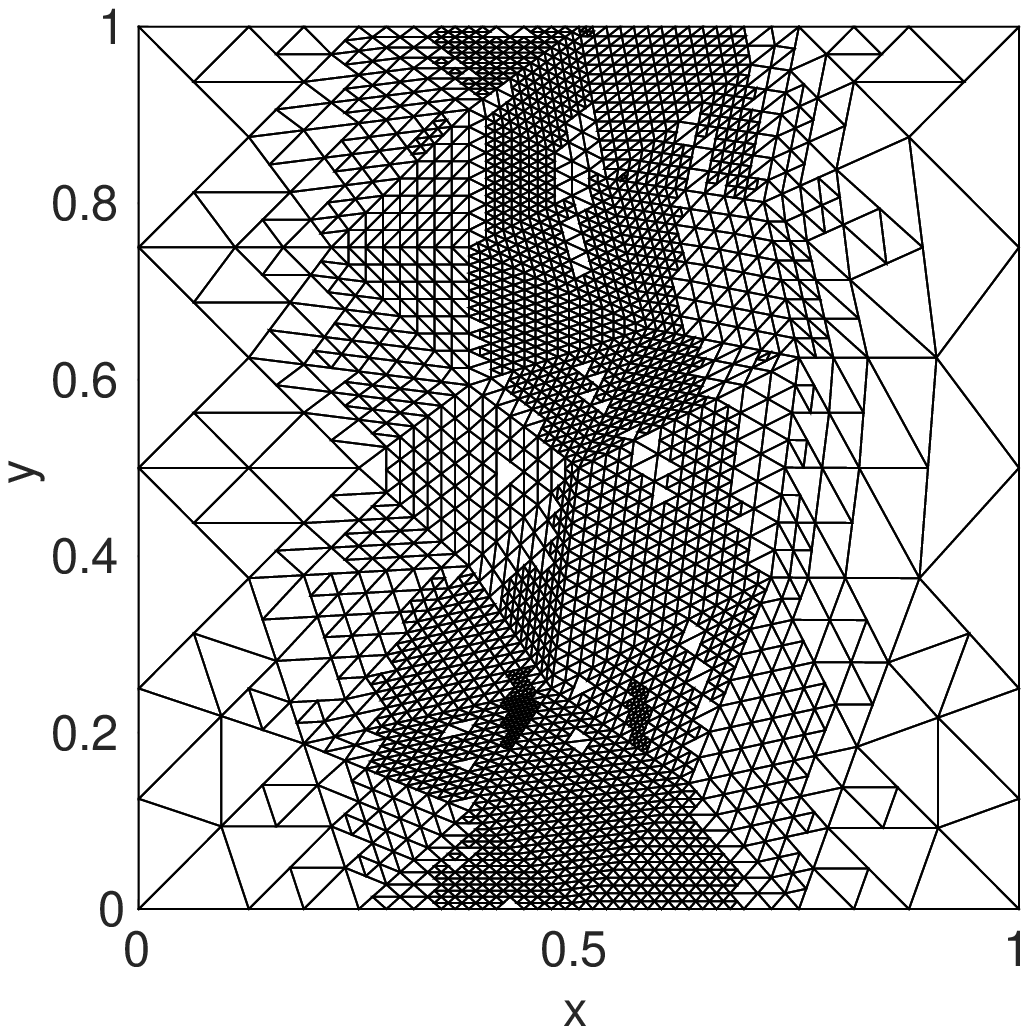}} \\
    \subfloat[Coarse ($hp$-refinement)]{\label{fig:square:mesh:hp:coarse}\includegraphics[width=0.45\textwidth]{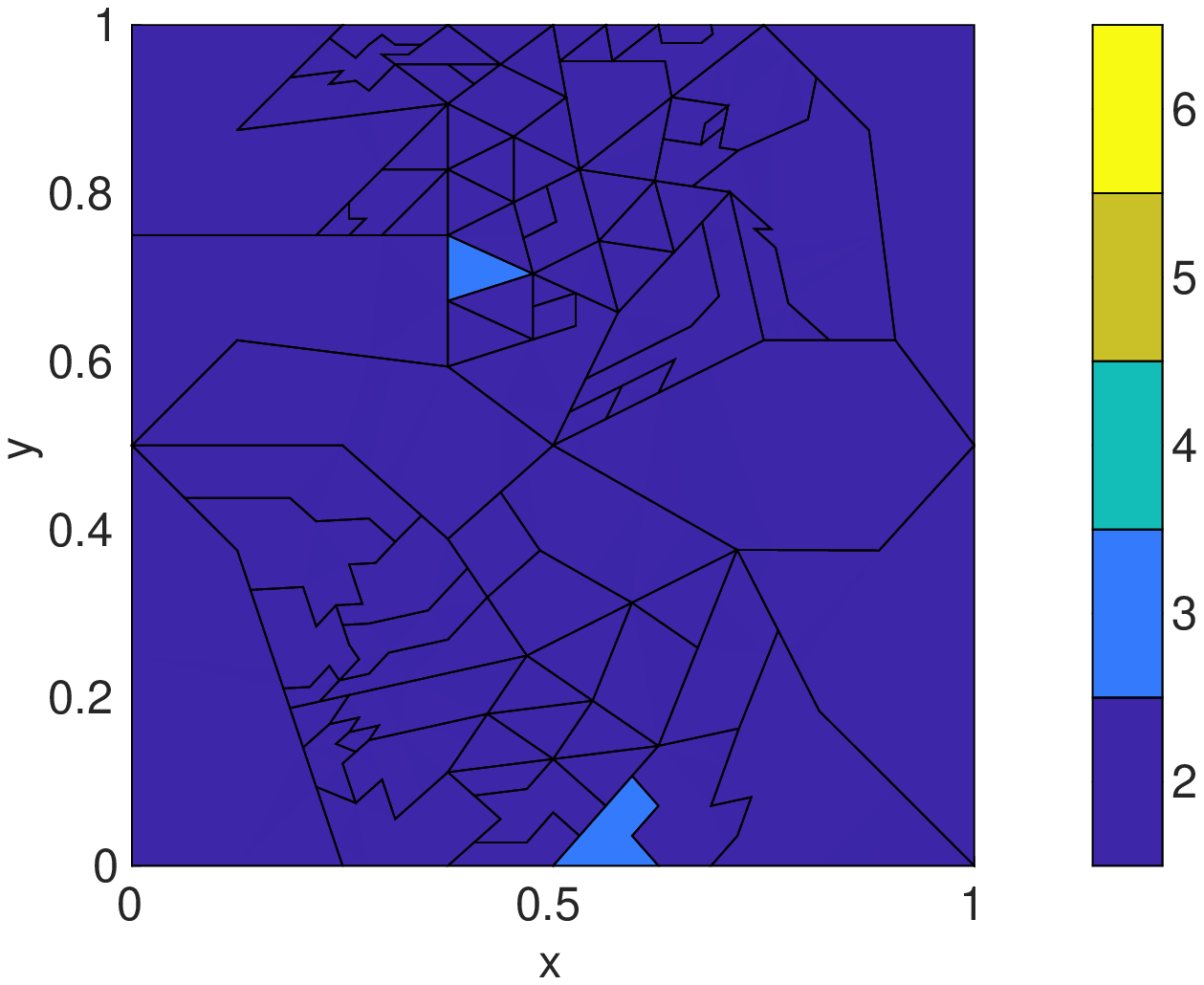}}
    \subfloat[Fine ($hp$-refinement)]{\label{fig:square:mesh:hp:fine}\includegraphics[width=0.45\textwidth]{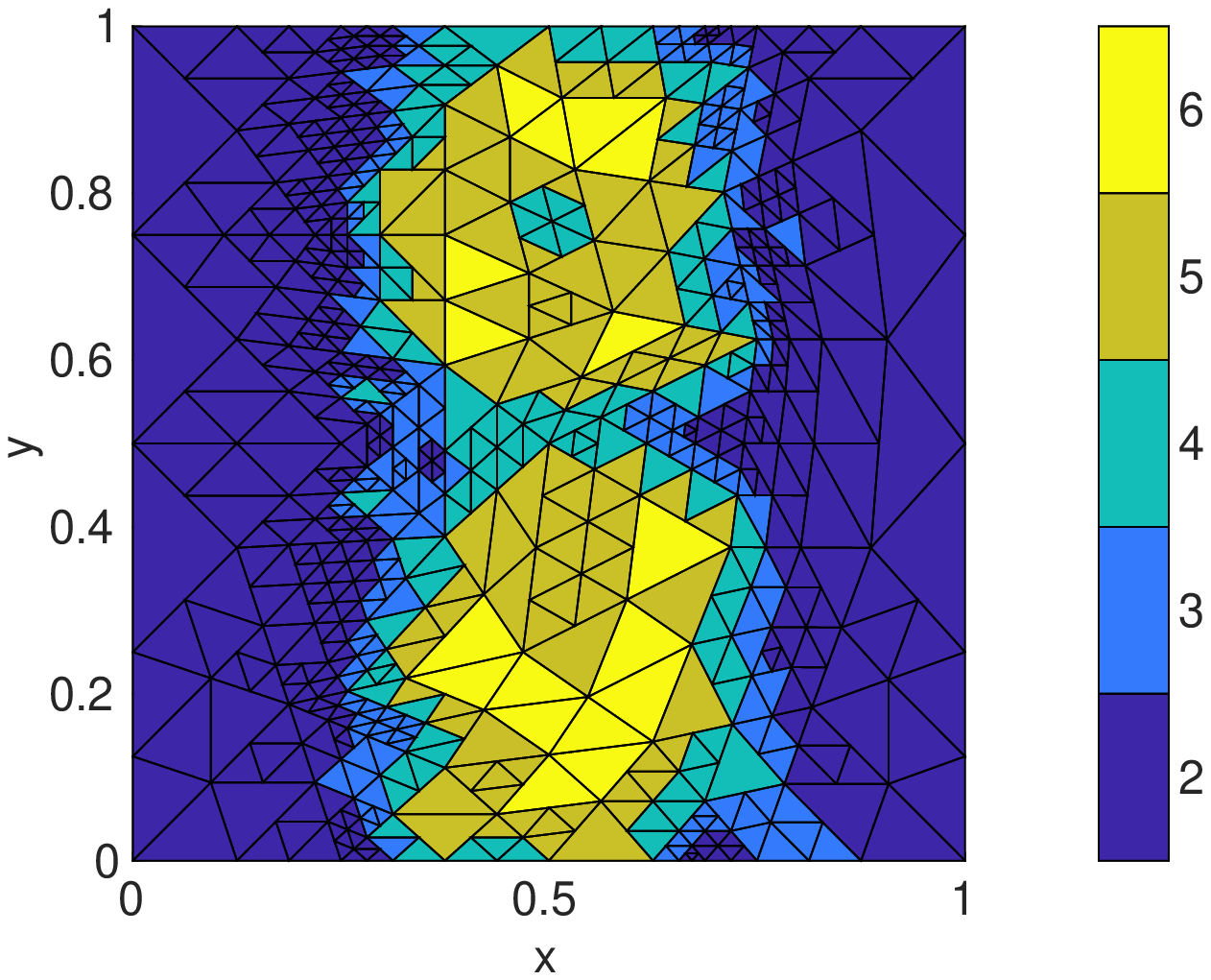}}
    \caption{Example 1. Coarse and fine meshes after $8$ \protect\subref{fig:square:mesh:h:coarse}--\protect\subref{fig:square:mesh:h:fine}~$h$- and \protect\subref{fig:square:mesh:hp:coarse}--\protect\subref{fig:square:mesh:hp:fine}~$hp$-adaptive mesh refinements, respectively}
    \label{fig:square:mesh}
\end{figure}

Finally, in \figref{fig:square:mesh} we show the fine and coarse meshes after 8 $h$- and $hp$-adaptive refinements for the two-grid method using the weighted, cf. \algoref{algo:errors}, coarse mesh refinement strategy, where the shading indicates the polynomial degree for the $hp$-refinement case. We notice that the refinement is concentrated around the `hills' in the analytical solution for both meshes, with mostly $p$-refinement in the interior, as would be expect when employing the standard DGFEM. We note considerably less refinement in the coarse $hp$-mesh compared to the fine one. 

\subsection{Example 2: Singular solution}
\begin{figure}[t]
    \configfigure
    \subfloat[]{\label{fig:lshape:err}\includegraphics[width=0.45\textwidth]{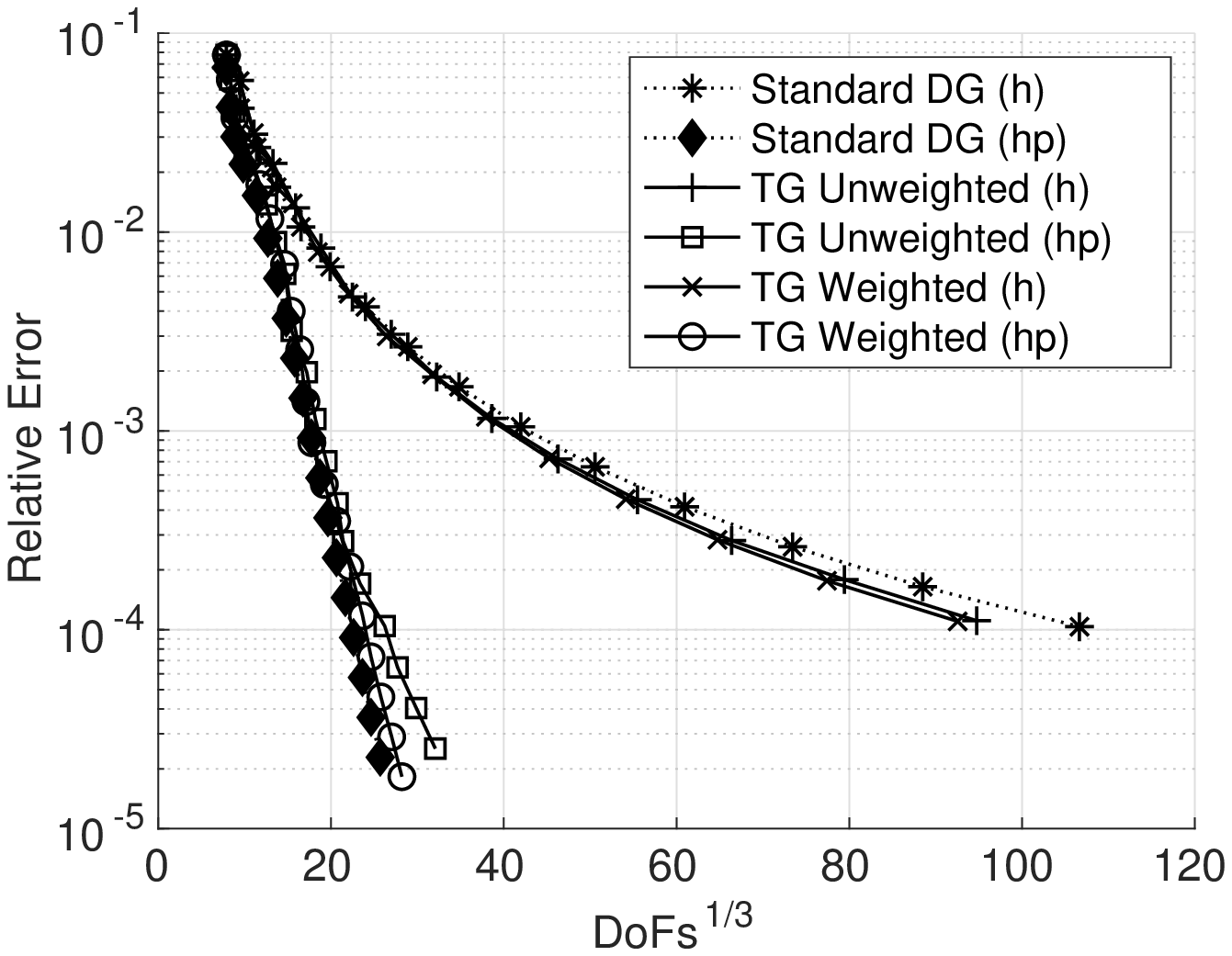}}
    \subfloat[]{\label{fig:lshape:eff}\includegraphics[width=0.45\textwidth]{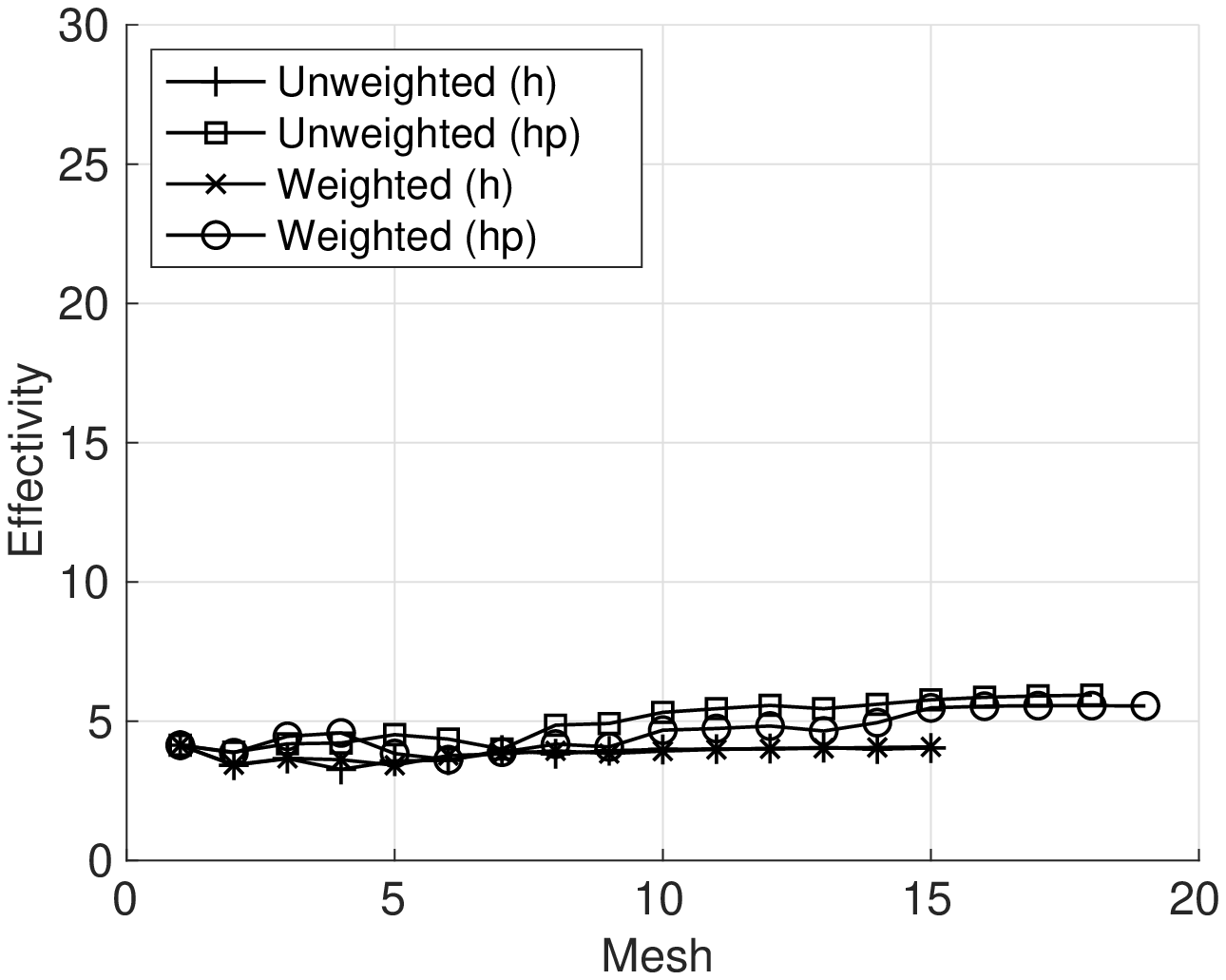}} \\
    \subfloat[]{\label{fig:lshape:dofs}\includegraphics[width=0.45\textwidth]{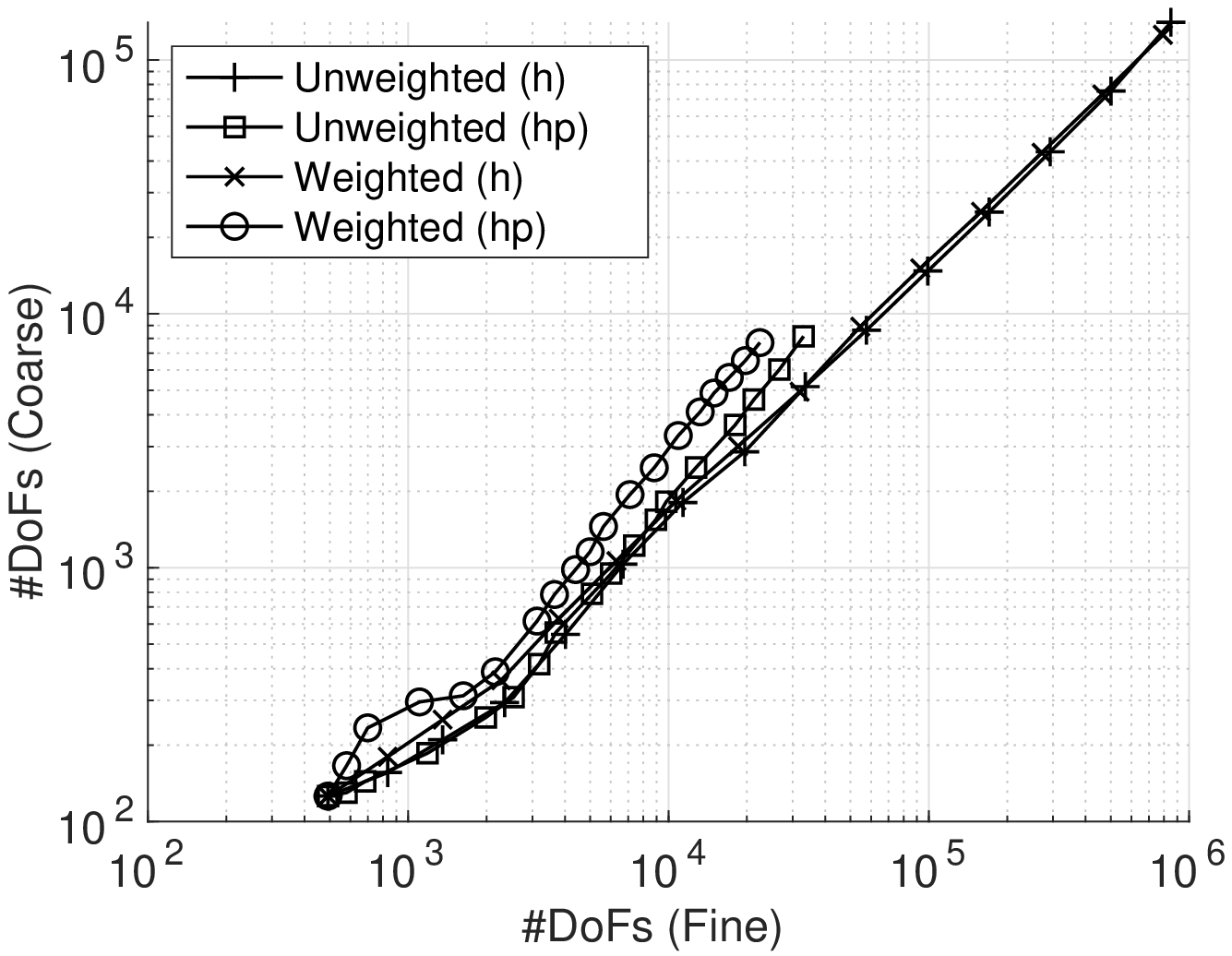}}
    \caption{Example 2. \protect\subref{fig:lshape:err}~Error in the DG norm with respect to the number of degrees of freedom for the standard method and the two-grid methods, using weighted and unweighted coarse mesh refinement, with $h$- and $hp$-refinement; \protect\subref{fig:lshape:eff}~Effectivity indices for the two-grid method with $h$- and $hp$-refinement; \protect\subref{fig:lshape:dofs}~Comparison of the number of degrees of freedom on the coarse and fine mesh for the two-grid methods}
    \label{fig:lshape}
\end{figure}

\begin{figure}[t]
    \configfigure
    \subfloat[$h$-refinement]{\label{fig:lshape:timing:h}\includegraphics[width=0.45\textwidth]{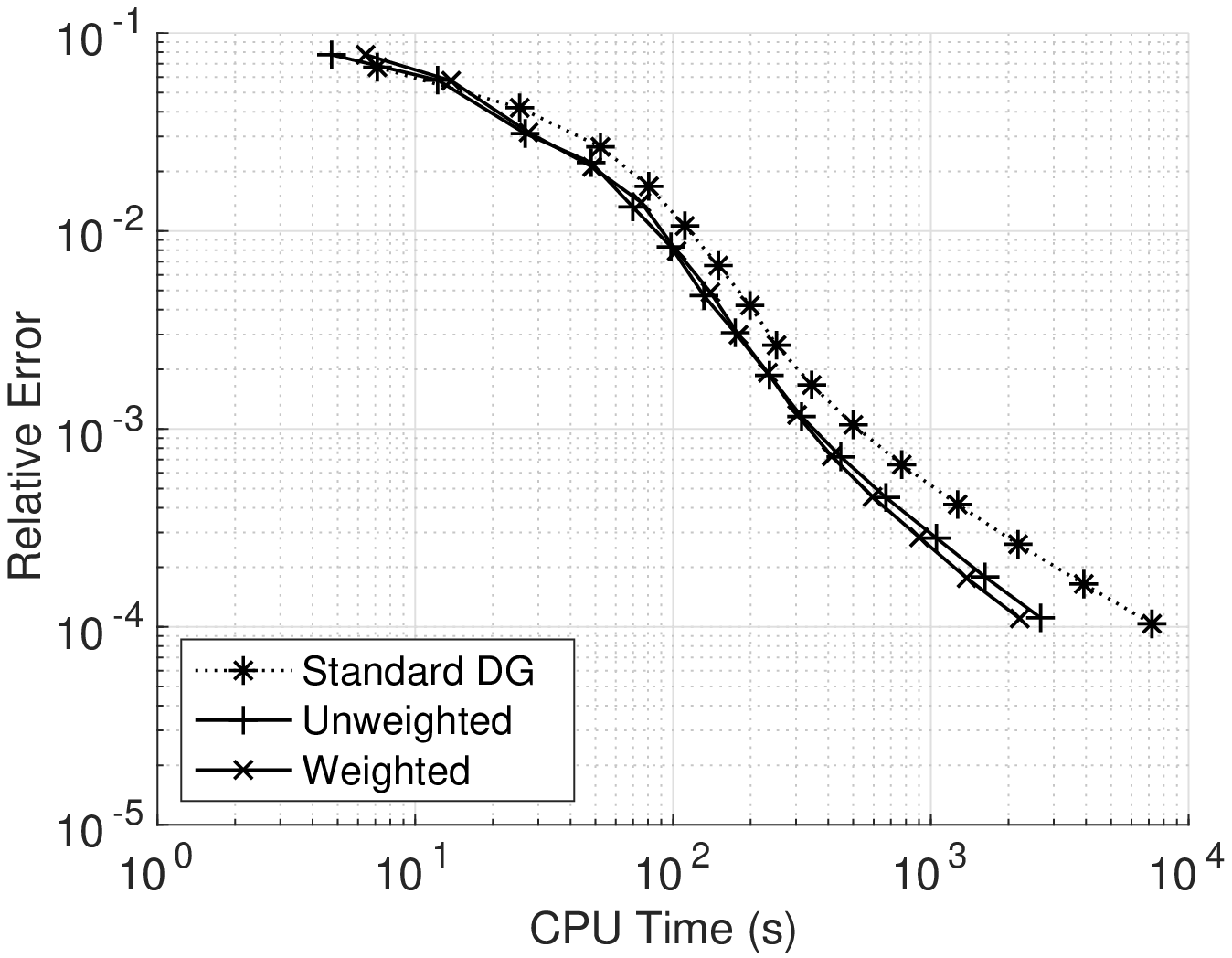}}
    \subfloat[$hp$-refinement]{\label{fig:lshape:timing:hp}\includegraphics[width=0.45\textwidth]{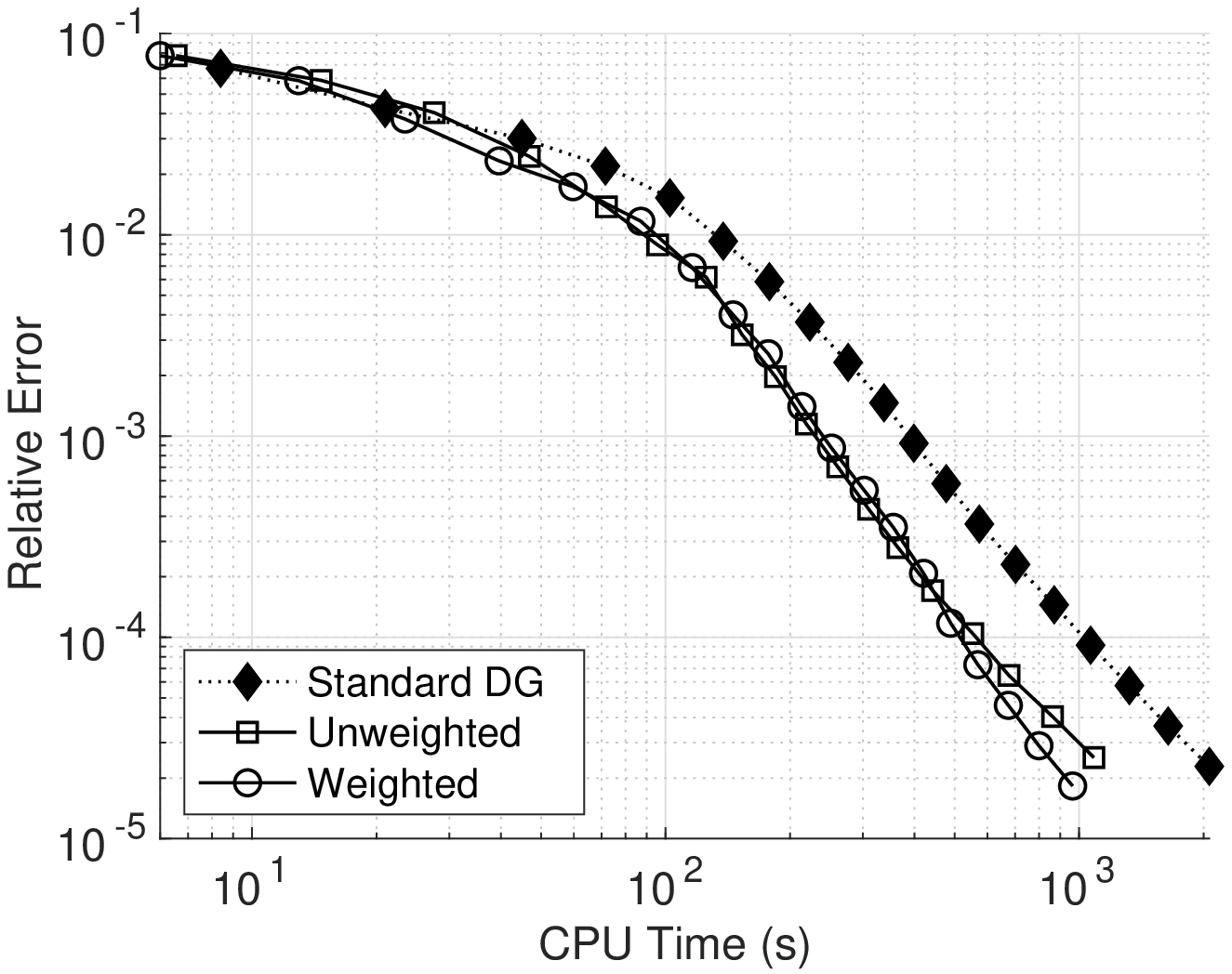}}
    \caption{Example 2. Error in the DG norm with respect to the cumulative CPU time for the standard method and the two-grid methods, using weighted and unweighted coarse mesh refinement, with \protect\subref{fig:lshape:timing:h}~$h$- and \protect\subref{fig:lshape:timing:hp}~$hp$-refinement}
    \label{fig:lshape:timing}
\end{figure}

In this example we consider the L-shaped domain $\Omega=(-1,1)^2\setminus[0,1)\times(-1,0]\subset\real^2$ and select the nonlinear coefficient to be
\[
\mu(\abs{\nabla u}) = 1 + \mathrm{e}^{-\abs{\nabla u}^2}.
\]
By writing $(r,\varphi)$ to denote the system of polar coordinates, we choose the forcing function $f$ and impose inhomogeneous boundary conditions such that the analytical solution to \eqref{eqn:quasilinear_eqn} is given by
\[
u(r,\varphi) = r^{\nicefrac23}\sin\left(\frac23\varphi\right).
\]
Note that $u$ is analytic in $\overline\Omega\setminus\{\vect{0}\}$, but $\nabla u$ is singular at the origin.

In \subfigref{fig:lshape}{err} we again present the comparison of the relative error measured in the DGFEM energy norm versus the third root of the number of degrees of freedom in the fine space $\dgspace$ for the standard formulation \eqref{eqn:standard_fem} and the two-grid formulation \eqref{eqn:tg_fem:coarse}--\eqref{eqn:tg_fem:fine} using both coarse mesh refinement strategies, \algoref{algo:naive} and \algoref{algo:errors}, when $h$- and $hp$-refinement is employed.  Here, we note that for $hp$-refinement the two two-grid methods again lead to a slight degradation in the error measured in the DGFEM norm, for a fixed number of degrees of freedom, when compared to the standard DGFEM. Additionally, we again observe that the two-grid DGFEM employing the weighted, cf. \algoref{algo:errors}, coarse mesh refinement strategy performs slightly better than the corresponding scheme exploiting the unweighted, cf. \algoref{algo:naive}, strategy. In the $h$-refinement setting, we actually observe the opposite behaviour: namely, that the two two-grid methods lead to a reduction in the error computed in the DGFEM norm, for a fixed number of degrees of freedom, when compared to the standard DGFEM, which is quite unexpected.
\subfigref{fig:lshape}{eff} again shows the effectivity indices for both two-grid refinement strategies using $h$- and $hp$-refinement; here, we observe that they are almost constant for all meshes indicating that the \emph{a posteriori} error bound overestimates the true error in a roughly consistent manner. \subfigref{fig:lshape}{dofs} again shows the coarse space degrees of freedom increasing at a slower rate compared to the corresponding quantity for the fine space for both two-grid DGFEMs employing either $h$- or $hp$-mesh refinement strategies; indeed, both methods result in a broadly similar number of coarse space degrees of freedom, although with slightly more coarse space degrees of freedom in the weighted $hp$-refinement case.

In \figref{fig:lshape:timing} we again compare the relative error computed in the DGFEM energy norm against the cumulative computation time for the standard DGFEM and both two-grid methods utilizing weighted and unweighted refinement of the coarse space. While we again notice a reduction in the DGFEM norm of error, for a given fixed computation time, when the two two-grid methods are employed compared to the standard DGFEM, this reduction is smaller than observed for the first example.

\begin{figure}[t!]
    \configfigure
    \subfloat[Coarse ($h$-refinement)]{\label{fig:lshape:mesh:h:coarse}\includegraphics[width=0.45\textwidth]{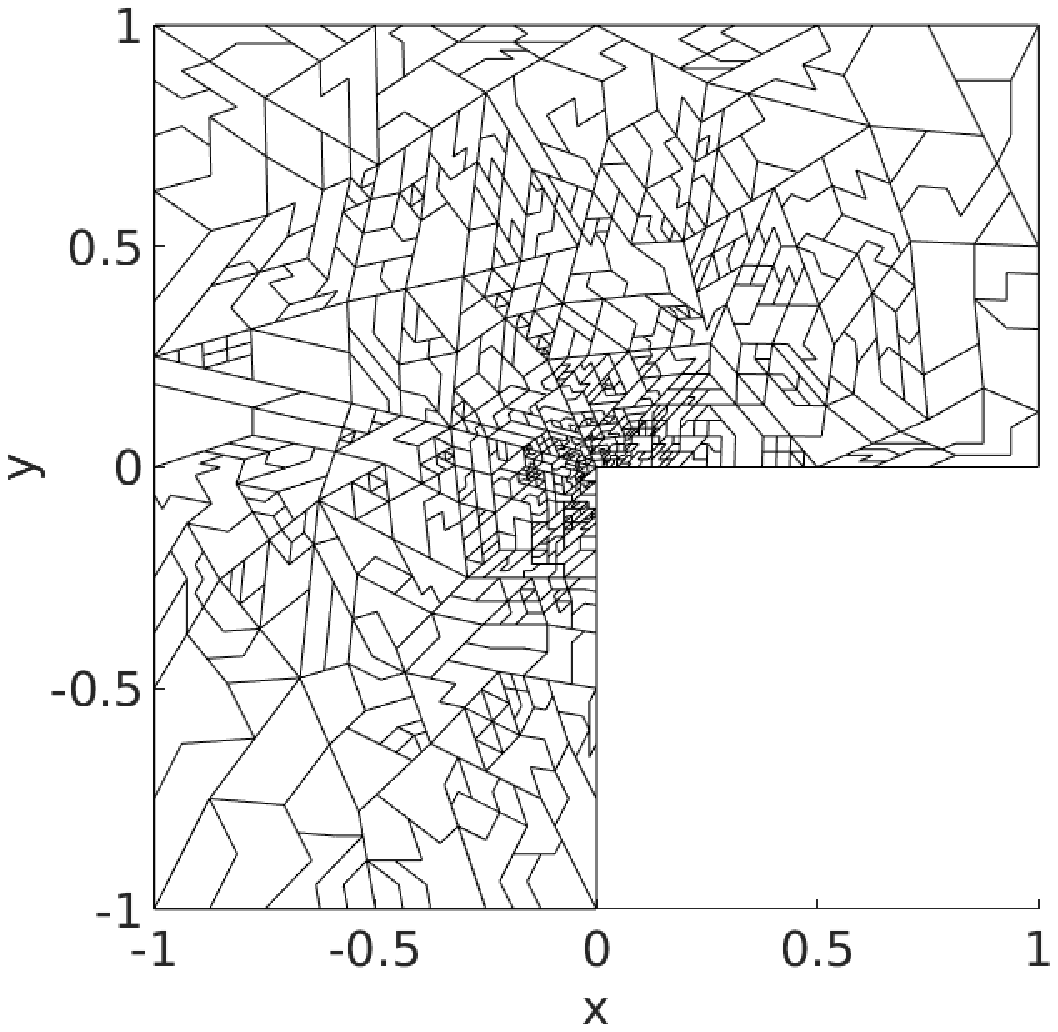}}
    \subfloat[Fine ($h$-refinement)]{\label{fig:lshape:mesh:h:fine}\includegraphics[width=0.45\textwidth]{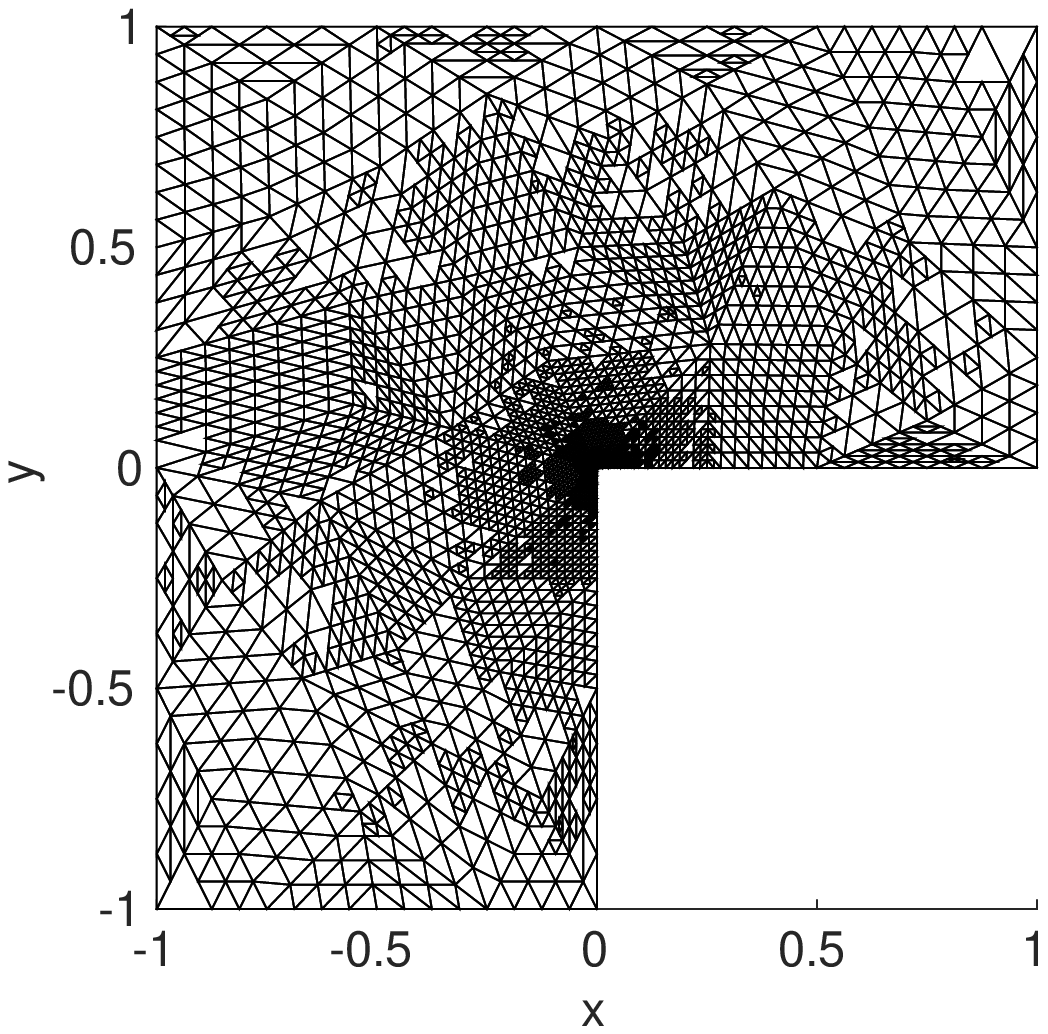}} \\
    \subfloat[Coarse ($hp$-refinement)]{\label{fig:lshape:mesh:hp:coarse}\includegraphics[width=0.45\textwidth]{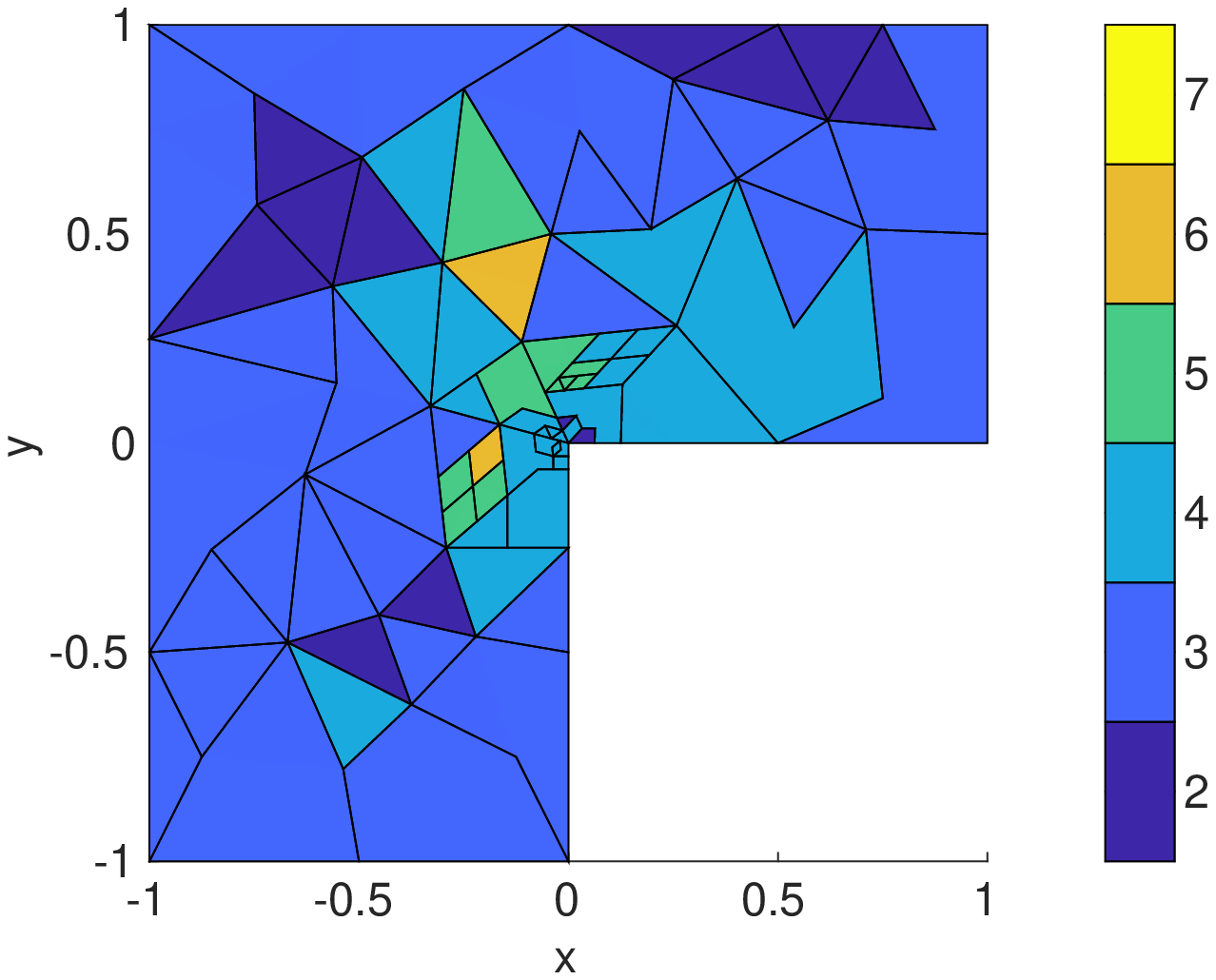}}
    \subfloat[Fine ($hp$-refinement)]{\label{fig:lshape:mesh:hp:fine}\includegraphics[width=0.45\textwidth]{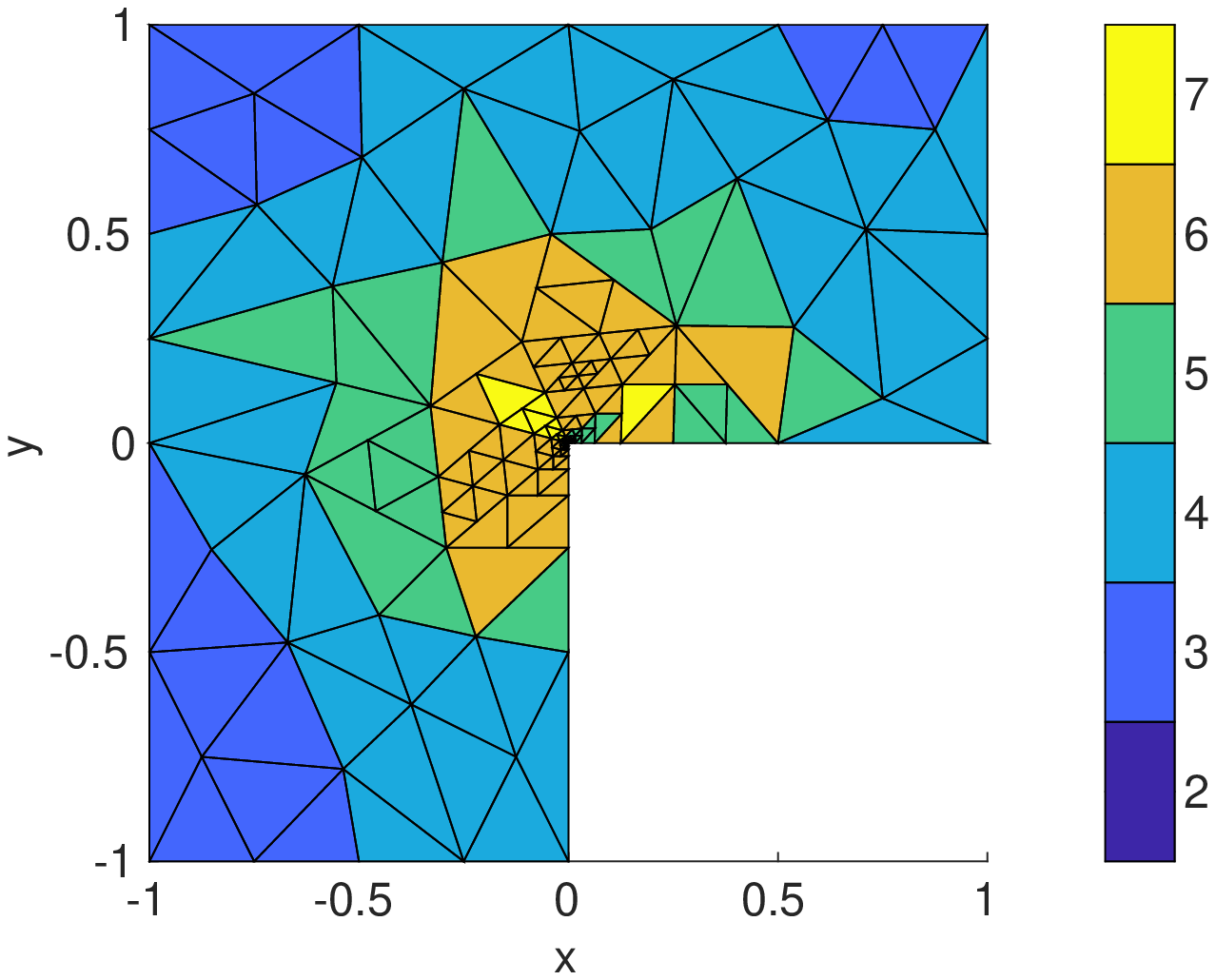}}
    \caption{Example 2. Coarse and fine meshes after $8$ \protect\subref{fig:lshape:mesh:h:coarse}--\protect\subref{fig:lshape:mesh:h:fine}~$h$- and \protect\subref{fig:lshape:mesh:hp:coarse}--\protect\subref{fig:lshape:mesh:hp:fine}~$hp$-adaptive mesh refinements, respectively}
    \label{fig:lshape:mesh}
\end{figure}

In \figref{fig:lshape:mesh} we show the coarse and fine meshes after 8 $h$- and $hp$-adaptive mesh refinements for the two-grid method using the weighted \algoref{algo:errors} coarse mesh refinement strategy. Here, we notice that for both coarse and fine meshes that the $h$-refinement is concentrated around the singularity at the re-entrant corner, with bands of $p$-refinement around this. We notice considerably more refinement on the coarse mesh compared with the previous example, caused by the method needing to resolve the singularity on both meshes.

\subsection{Example 3: 3D singular solution}
\begin{figure}[t!]
    \configfigure
    \subfloat[]{\label{fig:fichera:err}\includegraphics[width=0.45\textwidth]{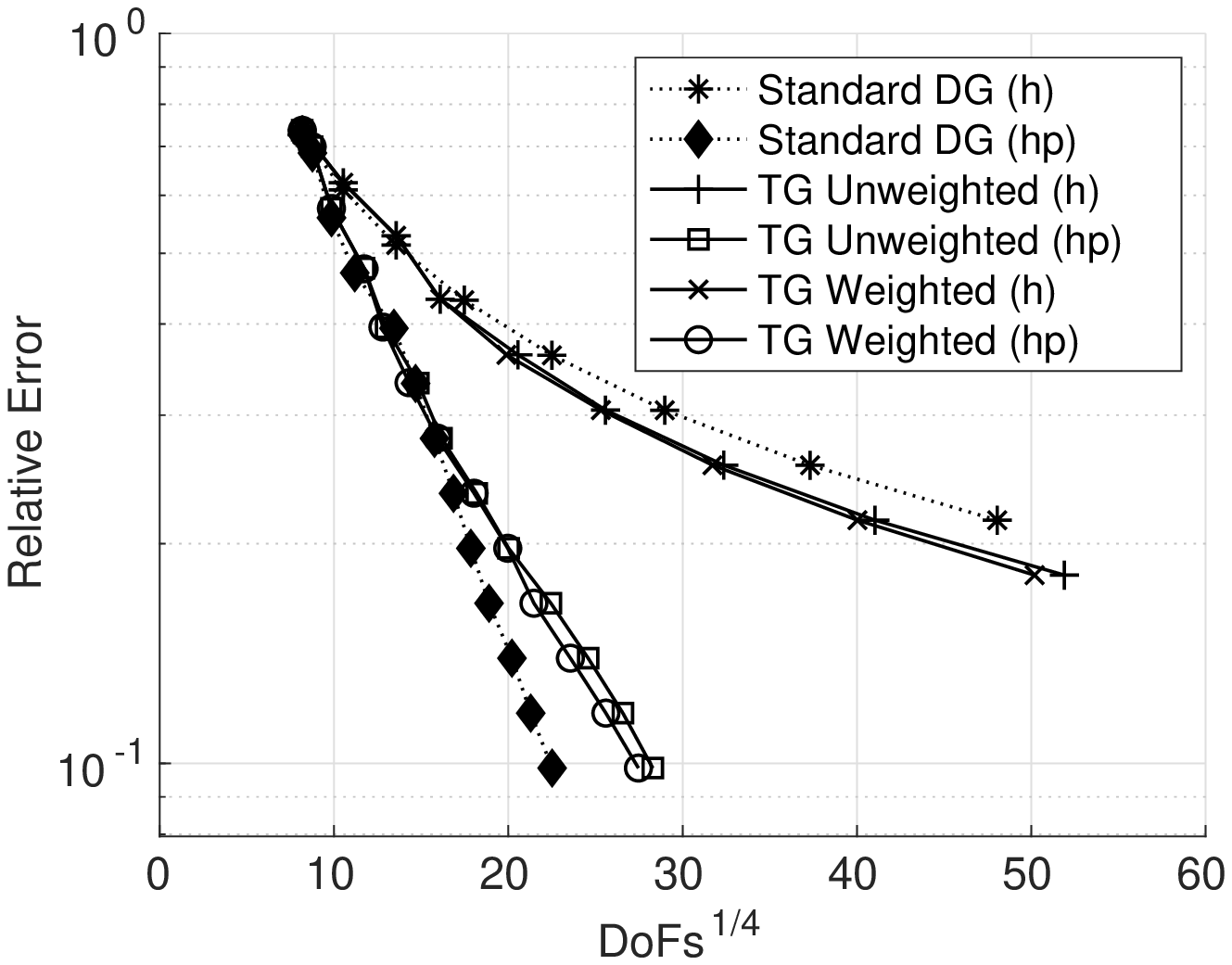}}
    \subfloat[]{\label{fig:fichera:eff}\includegraphics[width=0.49\textwidth]{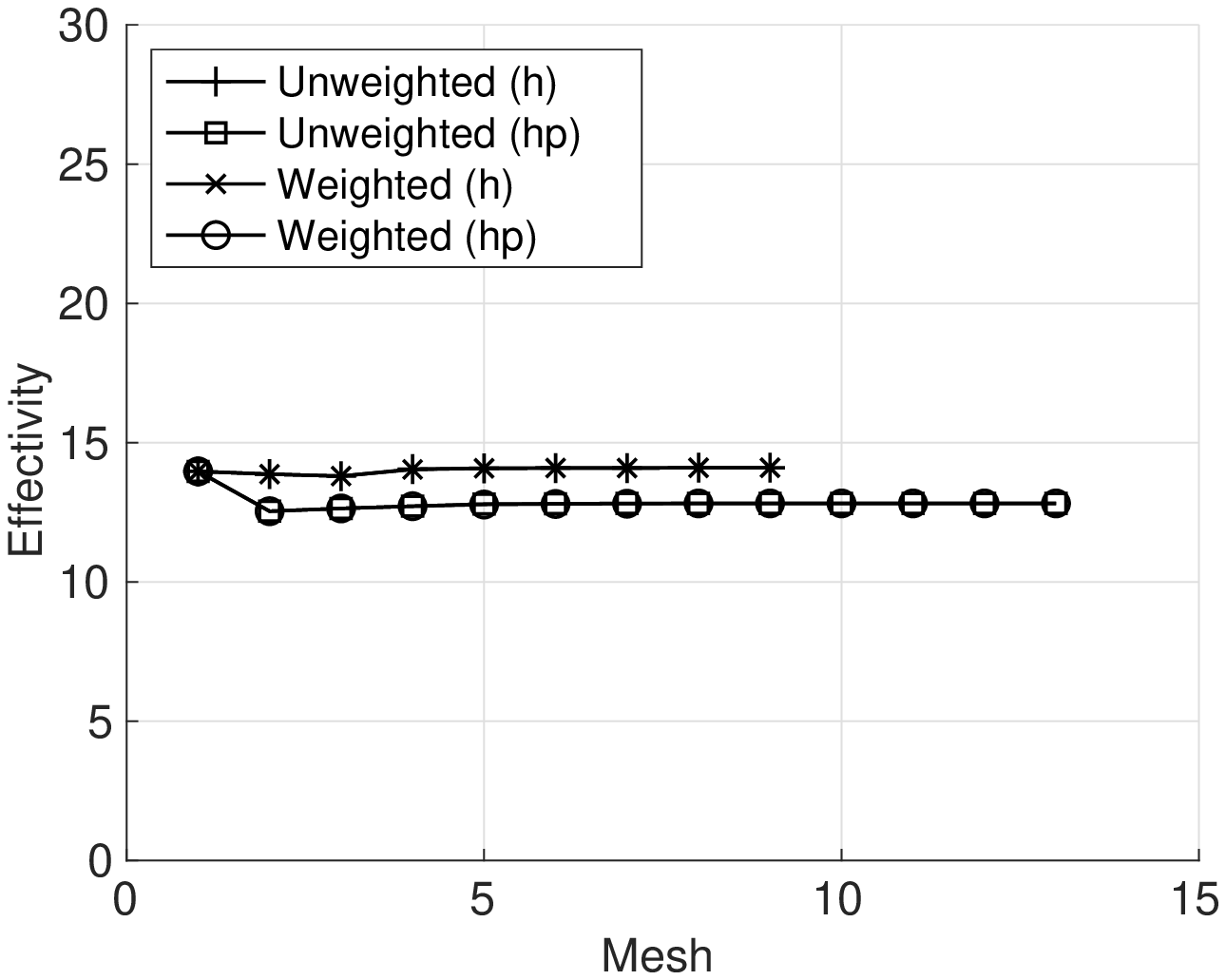}} \\
    \subfloat[]{\label{fig:fichera:dofs}\includegraphics[width=0.49\textwidth]{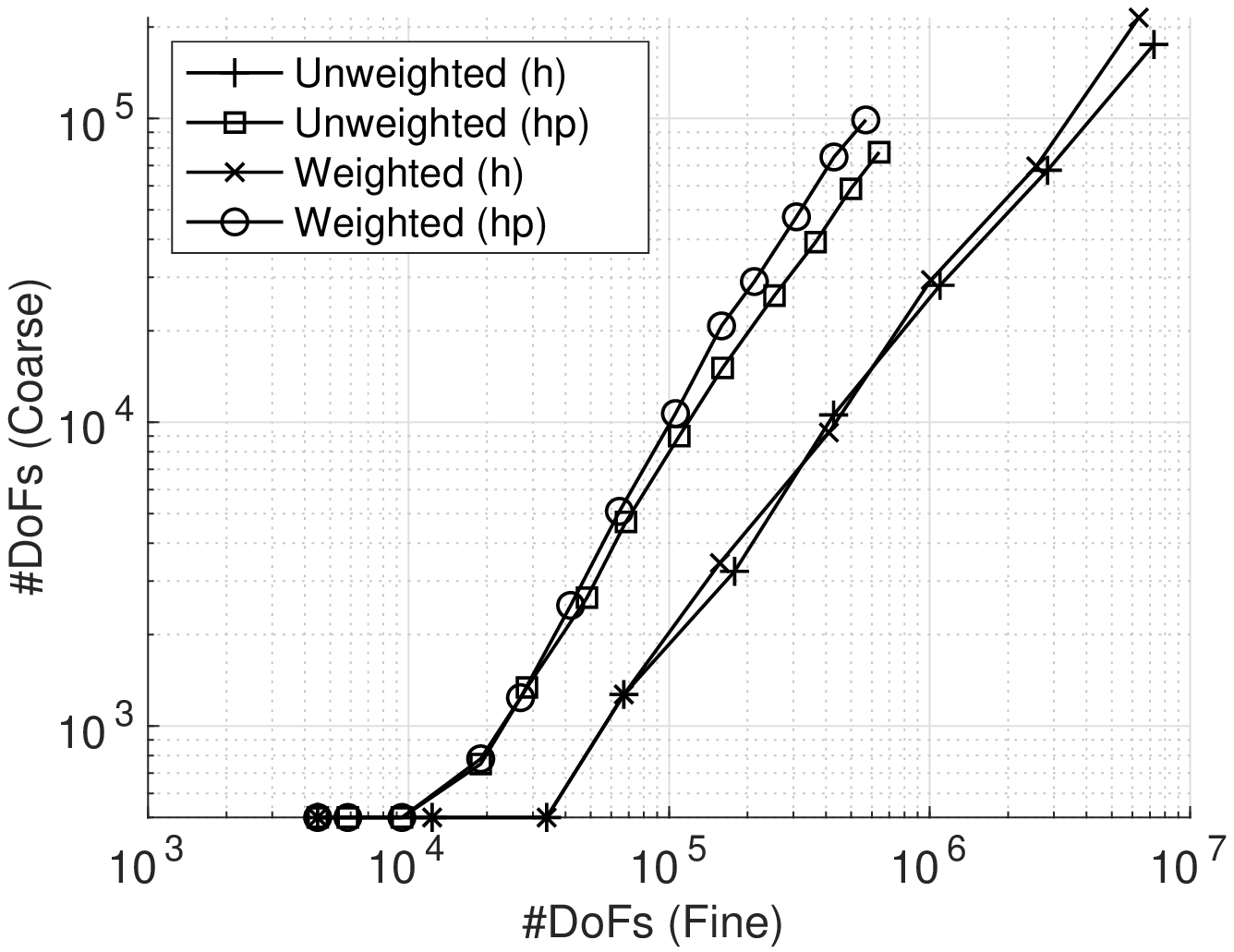}}
    \caption{Example 3. \protect\subref{fig:fichera:err}~Error in the DG norm with respect to the number of degrees of freedom for the standard method and the two-grid methods, using weighted and unweighted coarse mesh refinement, with $h$- and $hp$-refinement; \protect\subref{fig:fichera:eff}~Effectivity indices for the two-grid method with $h$- and $hp$-refinement; \protect\subref{fig:fichera:dofs}~Comparison of the number of degrees of freedom on the coarse and fine mesh for the two-grid methods}
    \label{fig:fichera}
\end{figure}

Finally, we consider a three-dimensional problem; to this end, we let $\Omega$ be the Fichera corner $(-1,1)^3\setminus[0,1)^3\subset\real^3$, use the nonlinearity \eqref{eqn:simple_mu} from the first example and select $f$ and a suitable inhomogeneous boundary condition such that the analytical solution to \eqref{eqn:quasilinear_eqn} is given by
\[
u(x,y,z) = (x^2+y^2+z^2)^{\nicefrac{q}2},
\]
where $q\in\real$. From \cite{Beilina} we note that for $q\geq-\nicefrac12$ the solution satisfies $u\in H^1(\Omega)$; in this case we select $q=-\nicefrac14$ as in \cite{Zhu}. We note that this gives a singularity at the re-entrant corner.

In \subfigref{fig:fichera}{err} we compare the relative error measured in the DGFEM norm with the fourth root, cf. \cite{Zhu}, of the number of degrees of freedom in $\dgspace$ for each of three methods considered in the previous examples, when $h$- or $hp$-refinement is employed.  We again notice that for $hp$-refinement we obtain exponential convergence, with a slightly slower rate when the two-grid methods are employed compared to the standard DGFEM; in the $h$-refinement setting the two two-grid methods lead to a reduction in the computed error, for a fixed number of degrees of freedom, when compared the standard DGFEM. \subfigref{fig:fichera}{eff} confirms that the \emph{a posteriori} error estimate again overestimates the error by a consistent amount in the sense that the effectivity indices for the two-grid methods employing both coarse mesh refinement strategies are roughly constant. Again, we observe that the coarse number of degrees of freedom grows slower than the number of degrees of freedom present in fine space, with a broadly similar number of degrees of freedom for both coarse mesh refinement strategies, cf. \subfigref{fig:fichera}{dofs}.

\begin{figure}[t!]
    \configfigure
    \subfloat[$h$-refinement]{\label{fig:fichera:timing:h}\includegraphics[width=0.45\textwidth]{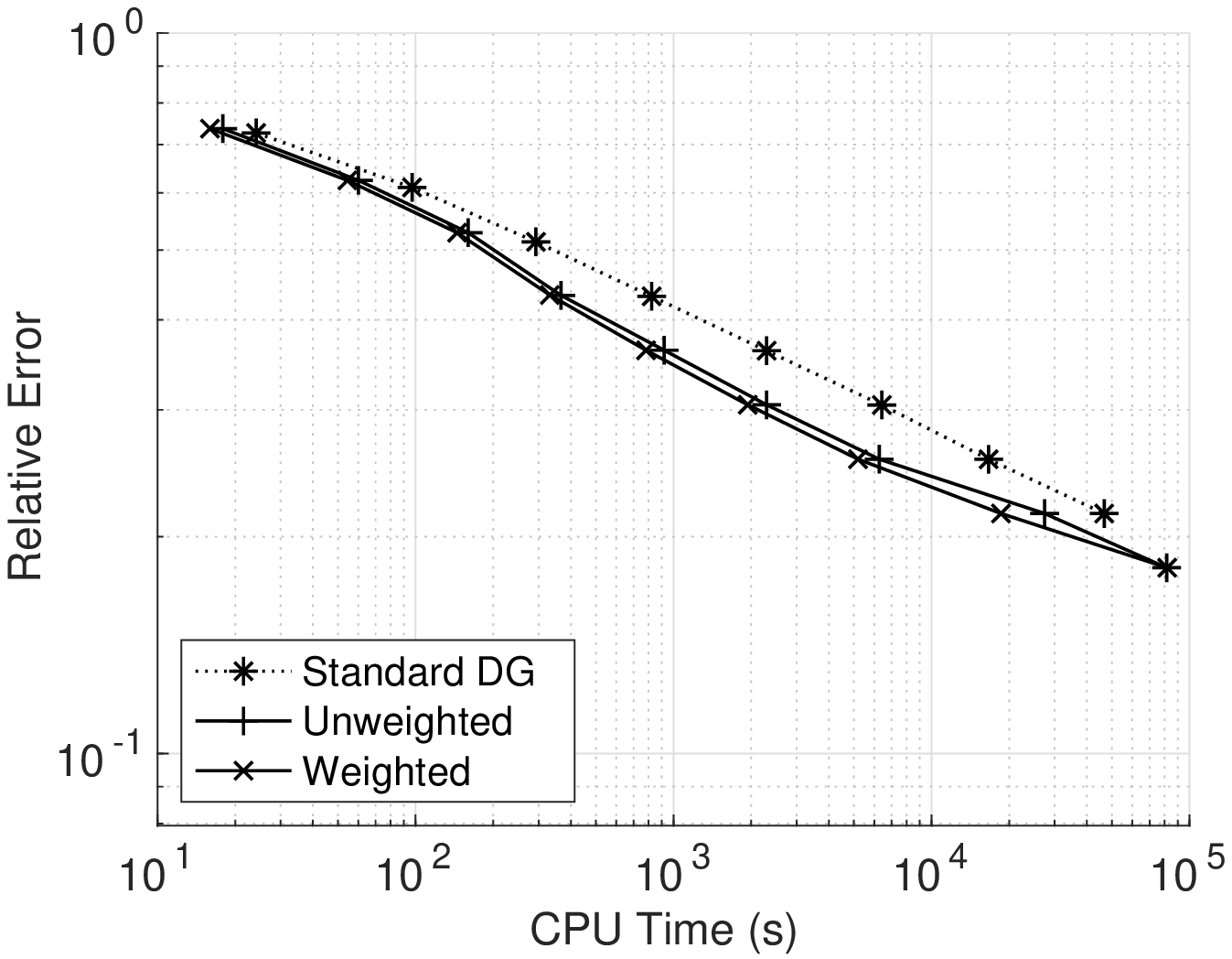}}
    \subfloat[$hp$-refinement]{\label{fig:fichera:timing:hp}\includegraphics[width=0.45\textwidth]{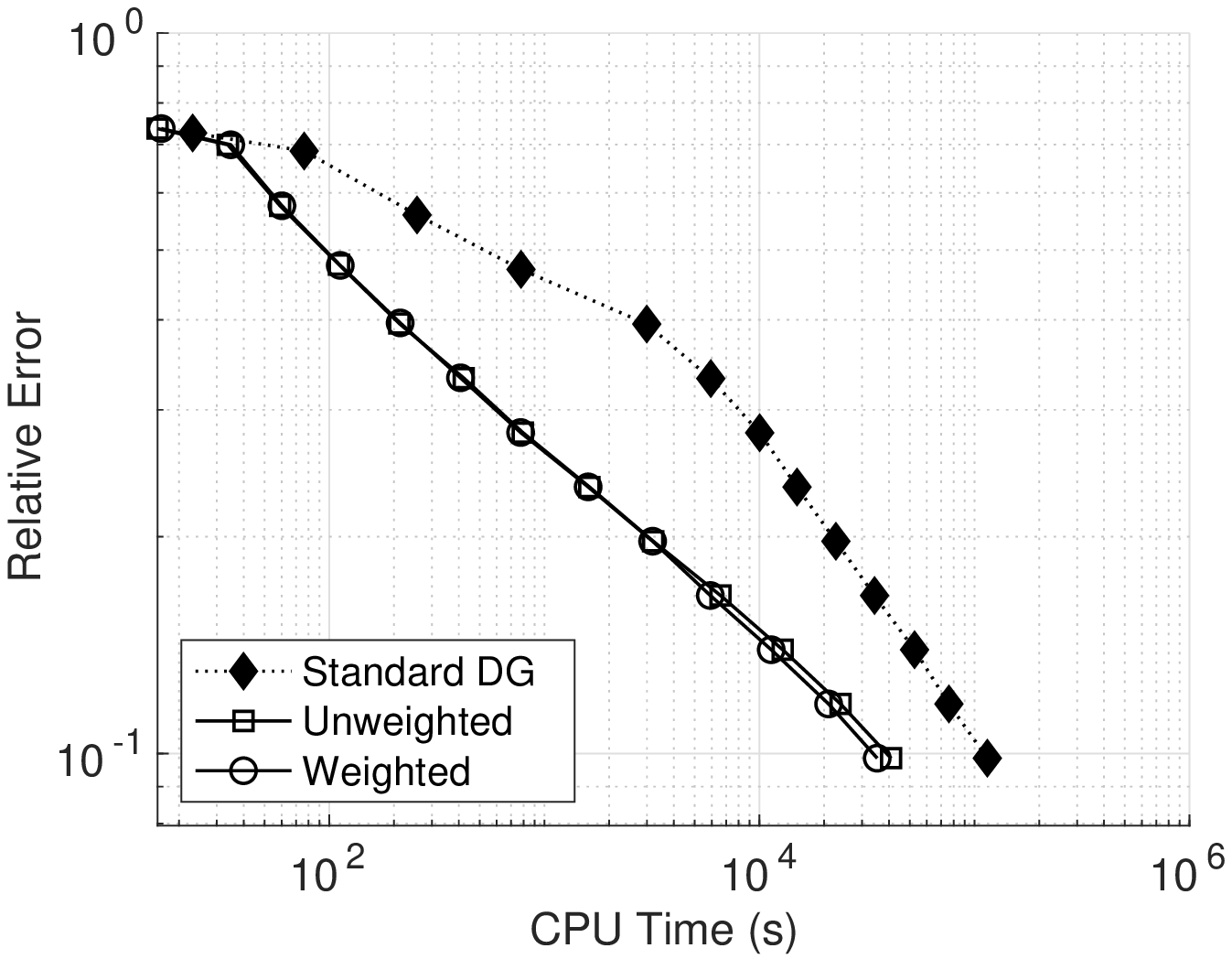}}
    \caption{Example 3. Error in the DG norm with respect to the cumulative CPU time for the standard method and the two-grid methods, using weighted and unweighted coarse mesh refinement, with \protect\subref{fig:fichera:timing:h}~$h$- and \protect\subref{fig:fichera:timing:hp}~$hp$-refinement}
    \label{fig:fichera:timing}
\end{figure}

We finally compare the relative error measured in the energy norm against the cumulative computation time taken for the standard DGFEM and the two two-grid methods employing both coarse mesh refinement strategies, for $h$- and $hp$-refinement; cf. \figref{fig:fichera:timing}. As for the previous example we notice a small reduction in the error for a fixed computation time when the two grid methods are employed compared with the standard DGFEM.

\section{Concluding remarks} \label{sec:concluding_remarks}

In this article, we have extended previous work on two-grid $hp$-version DGFEMs for the numerical approximation of second-order quasilinear boundary value problems of monotone type to the situation when general coarse meshes containing polytopic elements constructed by the agglomeration of fine mesh elements are employed. In particular, we have developed the \emph{a priori} error analyis for the polytopic coarse mesh approximation and developed algorithms for $hp$-adaptive refinement of the coarse mesh elements based on a computable \emph{a posteriori} error bound. This leads to fully adaptive black-box solver which can be used for the numerical approximation of nonlinear PDEs in an efficient manner. Indeed, our numerical experiments have highlighted that the computed error in the proposed two-grid method is generally similar in magnitude to the corresponding quantity computed based on employing a standard DGFEM formulation; however, the need to only solve a nonlinear system of equations on the coarse finite element space, with only a linear problem computed on the fine space, leads to significant reductions in the overall computation time when the former approach is employed. We have also shown that by weighting the refinement of the coarse mesh elements by the localized \emph{a posteriori} error indicators defined on the submesh partition that forms the coarse element, we are able to reduce the error compared to both the number of  degrees of freedom in the fine finite element space and the overall computation time. Further extensions of this work include the application to PDE problems with coefficients containing more general nonlinearities.

\section*{Acknowledgments}

SC has been supported by Charles University Research program No. UNCE/SCI/023 and the Czech Science Foundation (GA\v{C}R) project No. 20-01074S. PH acknowledges the financial support of the EPSRC under the grant EP/R030707/1.

\bibliographystyle{abbrv}
\bibliography{references}

\end{document}